\documentclass[12pt]{article}

\usepackage[dvipsnames]{xcolor}
\usepackage{amsmath,amsfonts, amsthm, amssymb}
\usepackage{array}
\usepackage{microtype}
\usepackage{graphicx}
\usepackage{tikz}
\usetikzlibrary{calc}
\usetikzlibrary{decorations.pathmorphing}
\usetikzlibrary{shapes.misc}

\tikzset{snake it/.style={decorate, decoration={snake, amplitude=0.5mm, segment length=3mm}}}
\tikzset{cross/.style={cross out, draw=black},
cross/.default={1pt}}

\newcommand{\arcto}[3]{
\draw[#3] let
  \p1 = ($#1-#2$),
  \n1 = {veclen(\x1,\y1)}
in
  #1 arc(180:0:\n1/2);
}

\newcommand{\ER}[1][1] {
\begin{tikzpicture}[scale=#1]
\coordinate (A) at (0,0);
\coordinate (B) at (2,0);
\draw [thick] (A) -- (B);
\node [circle,fill=white,inner sep=0.5pt] at (1,-0.07) {};
\end{tikzpicture}
}

\newcommand{\snake}[1][1] {
\begin{tikzpicture}[scale=#1]
\coordinate (A) at (0,0);
\coordinate (B) at (2,0);
\draw [semithick, densely dashed] (A) -- (B);
\node [circle,fill=white,inner sep=0.5pt] at (1,-0.07) {};
\end{tikzpicture}
}

\newcommand{\snakeO}[1][1] {
\begin{tikzpicture}[scale=#1]
\coordinate (A) at (0,0);
\coordinate (Q) at (1,0);
\coordinate (B) at (2,0);
\draw [semithick, densely dashed] (A) -- (B);
\node [circle, draw=black, fill=white, inner sep=2pt, line width=0.7pt] at (Q) {};
\node [cross, fill=black, inner sep=1.5pt, line width=1pt] at (Q) {};
\end{tikzpicture}
}
\newcommand{\snakeT}[1][1] {
\begin{tikzpicture}[scale=#1]
\coordinate (A) at (0,0);
\coordinate (Q) at (0.6,0);
\coordinate (R) at (1.4,0);
\coordinate (B) at (2,0);
\draw [semithick, densely dashed] (A) -- (B);
\node [circle, draw=black, fill=white, inner sep=2pt, line width=0.7pt] at (Q) {};
\node [cross, fill=black, inner sep=1.5pt, line width=1pt] at (Q) {};
\node [circle, draw=black, fill=white, inner sep=2pt, line width=0.7pt] at (R) {};
\node [cross, fill=black, inner sep=1.5pt, line width=1pt] at (R) {};
\end{tikzpicture}
}

\newcommand{\SCE}[1][1] {
\begin{tikzpicture}[scale=#1]
\coordinate (A) at (0,0);
\coordinate (B) at (2,0);
\draw [semithick, densely dashed] (A) -- (B);
\node [cross, fill=black, inner sep=1.5pt, line width=1pt] at (A) {};
\node [cross, fill=black, inner sep=1.5pt, line width=1pt] at (B) {};
\arcto{(A)}{(B)}{semithick}
\end{tikzpicture}
}
\def\Real{{\mathbb R}}

\def\bbZ{{\mathbb Z}}
\def\bbN{{\mathbb N}}

\def\Torus{{\mathbb T}}

\def\Expec{{\mathbb E}}
\def\Prob{{\mathbb P}}
\def\eps{{\varepsilon}}
\def\One{{\mathbf{1}}}
\DeclareMathOperator{\trace}{tr}

\newcommand{\diff}{\mathop{}\!\mathrm{d}}
\newcommand{\Id}{{\operatorname*{Id}}}
\newcommand{\Hilb}{{\operatorname*{H}}}
\newcommand{\Ft}[1]{{\widehat{#1}}}

\newcommand{\Impt}{\operatorname*{Im}}

\newcommand{\Op}{\operatorname*{Op}}

\newcommand{\Exercise}[1]{\noindent\textbf{Exercise.  }#1}
\newcommand{\Sol}[1]{}

\usepackage[style=alphabetic, backend=biber,maxbibnames=99]{biblatex}
\def\bra#1{\mathinner{\langle{#1}|}}
\def\ket#1{\mathinner{|{#1}\rangle}}
\def\braket#1{\mathinner{\langle{#1}\rangle}}

\newcommand{\mcal}[1]{{\mathcal{#1}}}

\newcommand{\ldom}{\prec}

\numberwithin{equation}{section}

\newtheorem{theorem}{Theorem}

\numberwithin{theorem}{section}
\newtheorem{lemma}[theorem]{Lemma}

\newtheorem{corollary}[theorem]{Corollary}
\newtheorem{proposition}[theorem]{Proposition}

\def\lsim{{\,\lesssim\,}}

\title{Lecture notes on Quantum Diffusion and Random Matrix Theory}
\author{Felipe Hern{\'a}ndez}
\date{\today}

\begin{document}
\maketitle

\begin{abstract}
In joint work with Adam Black and Reuben Drogin~\cite{black2025tail,black2025self}, we develop a new approach to understanding the diffusive limit of the random Schrodinger equation based on ideas taken from random matrix theory.  These lecture notes
present the main ideas from this work in a self-contained and simplified presentation.  The lectures were given at the summer school ``PDE and Probability'' at Sorbonne Universit\'e from June 16-20, 2025.
\end{abstract}

\tableofcontents

\section{Introduction}
The goal of these lectures is to say what we can about the random Schr{\"o}dinger equation
in the weak coupling limit:
\begin{equation}
\label{eq:RSE}
i\partial_t \psi = \Delta \psi + \lambda V \psi.
\end{equation}
Above, $\Delta$ is the Laplacian, $d\geq 2$, $\lambda$ is a small coupling
parameter $\lambda\ll 1$, and $V$ is a random potential.  One can consider~\eqref{eq:RSE} on $\Real^d$
or $\bbZ^d$.  On $\Real^d$ one can take $V$ to be a stationary Gaussian field (for example), and
on $\bbZ^d$ it is simplest to take $V$ to have independent standard Gaussian entries, and $\Delta$
to be the nearest-neighbor Laplacian.

The motivation for studying~\eqref{eq:RSE} is that it is a model for studying wave transport in random
media.  Indeed, the Schrodinger equation is the simplest example of a dispersive PDE, and the term
 $\lambda V$ is the simplest kind of random perturbation that can be made to it.
Examples of wave transport in random media arise in telecommunications, geological imaging, and condensed matter physics (see~\cite{bal2010kinetic} for a survey of the field of waves in random media).
This last example is closest to the specific model described above -- in~\cite{anderson1958absence}, Anderson introduced~\eqref{eq:RSE} as a model for electron transport in disordered materials.
More specifically, Anderson considered a model on the lattice $\bbZ^d$ where $\Delta$ represents a
discrete ``hopping'' term\footnote{Anderson actually considered much more general hopping terms which could be nonlocal and nonuniform.} between lattice sites and $\lambda V$ is the potential associated to each site.
In this case the Hamiltonian $H=\Delta + \lambda V$ describes the effective energy of a single electron
in a disordered material, and thus encodes the electrical properties of this material.
For broader context on wave scattering and localization in disordered media, see the monographs~\cite{akkermans2007mesoscopic,sheng2007introduction}.

The simplest electrical property is the conductivity, and it is natural to ask whether $H$ describes
an insulator or a conductor.  A direct way to model
this mathematically (as considered already in Anderson's original work)
is to observe the dispersion of an initially localized wavefunction
$\psi_0\in\ell^2(\bbZ^d)$  -- for example one may take $\psi_0 = \ket{0}$, the wavefunction localized
at the site at the origin.  In this case one is interested in the mean square displacement of the wavefunction at time $t$, for $\psi_t$ solving~\eqref{eq:RSE}.  That is, one defines
\begin{equation}
\label{eq:rdef}
r^2(t) := \sum_{x\in\bbZ^d} |x|^2 |\psi_t(x)|^2.
\end{equation}
In~\cite{anderson1958absence}, Anderson argued that for sufficiently large $\lambda$ there is some constant $r_{\rm max}$ depending on the initial data $\psi_0$ and the potential $\lambda V$ such that $r(t) \leq r_{\rm max}$ for all $t$.
The physical interpretation of this fact is that sufficiently disordered materials are
\textit{insulators}.

That $r(t)$ is bounded at high disorder is derived a consequence of the fact that the operator $H$ has
a pure point spectrum of orthonormal eigenfunctions, each exponentially localized to a finite interval.
Anderson's original paper inspired an entire subfield of condensed matter physics characterizing the localization of eigenfunctions of random Schrodinger operators.  A thorough survey of localization is outside the scope of these lectures
(see~\cite{aizenman2015random,stollmann2001caught} for more comprehensive treatments),
but we can summarize what is known rigorously as follows:
\begin{itemize}
\item In $d=1$, the operator $H$ has a complete orthonormal basis of localized eigenfunctions, and therefore $r(t) \leq r_{\rm max} < \infty$ almost surely~\cite{gol1977pure,kunz1980spectre}.
.  In the weak coupling limit $\lambda \ll1$, it is the case that
$r_{\rm max} \simeq \lambda^{-2}$~\cite{pastur1980spectral,CL90}.
\item In any $d\geq 1$ and any $\lambda >0$, there exist localized eigenfunctions of $H$ near the spectral edges~\cite{frohlich1983absence}.
\end{itemize}
A conspicuous gap in the mathematically rigorous theory of localization is what happens to the \textit{bulk} of the spectrum in $d\geq 2$.  It is conjectured that in $d\geq 3$ there is a ``metal-insulator'' transition in the spectrum between localized eigenfunctions and pure point spectrum near the edges and a continuous spectrum consisting of ``extended'' or delocalized eigenstates in the bulk.  In $d=2$ it is instead conjectured that the bulk consists of localized eigenfunctions with localization length scale on the order $e^{c\lambda^{-2}}$, as predicted in~\cite{abrahams1979scaling}.  These conjectures appear in the list of Simon's problem's~\cite{simon1984fifteen,simon2000schrodinger}.

In terms of bounds on $r(t)$, the above results and conjectures correspond to the following bounds
on $r(t)$:
\[
\sup_{t<\infty} r(t) = r_{\rm max} \simeq \begin{cases} \lambda^{-2}, & d=1 \\
e^{\lambda^{-2}}, & d=2 \\
+\infty, &d\geq 3.
\end{cases}
\]
The conjectured values for $r_{\rm max}$, the extended states conjecture, and the existence of the mobility edge seem to be far out of the reach of current methods.  A more modest goal is to
try to characterize $r(t)$ for finite (as opposed to infinite) times.  This is the goal of these
lecture notes.   There are two heuristics that determine the behavior of $r(t)$:
\begin{enumerate}
\item For times $t\ll\lambda^{-2}$, the evolution $e^{-itH}$ behaves like the free evolution $e^{-it\Delta}$.
\item The effect of the potential $V$ is to scatter the wavefunction, and each scattering event is independent.
\end{enumerate}
The combination of these heuristics suggests that $\psi_t(x)$ should be some random superposition
of random walk paths of $\lambda^2t$ steps and step size $\lambda^{-2}$.  Therefore, one expects that
$|\psi_t|^2$ resembles a Gaussian and $r(t) \approx \lambda^{-1}t^{1/2}$.  See Figure~\ref{fig:numerical} for a numerical simulation.

\begin{figure}
\centering
\begin{tabular}{m{7cm}m{7cm}}
\includegraphics[scale=0.7]{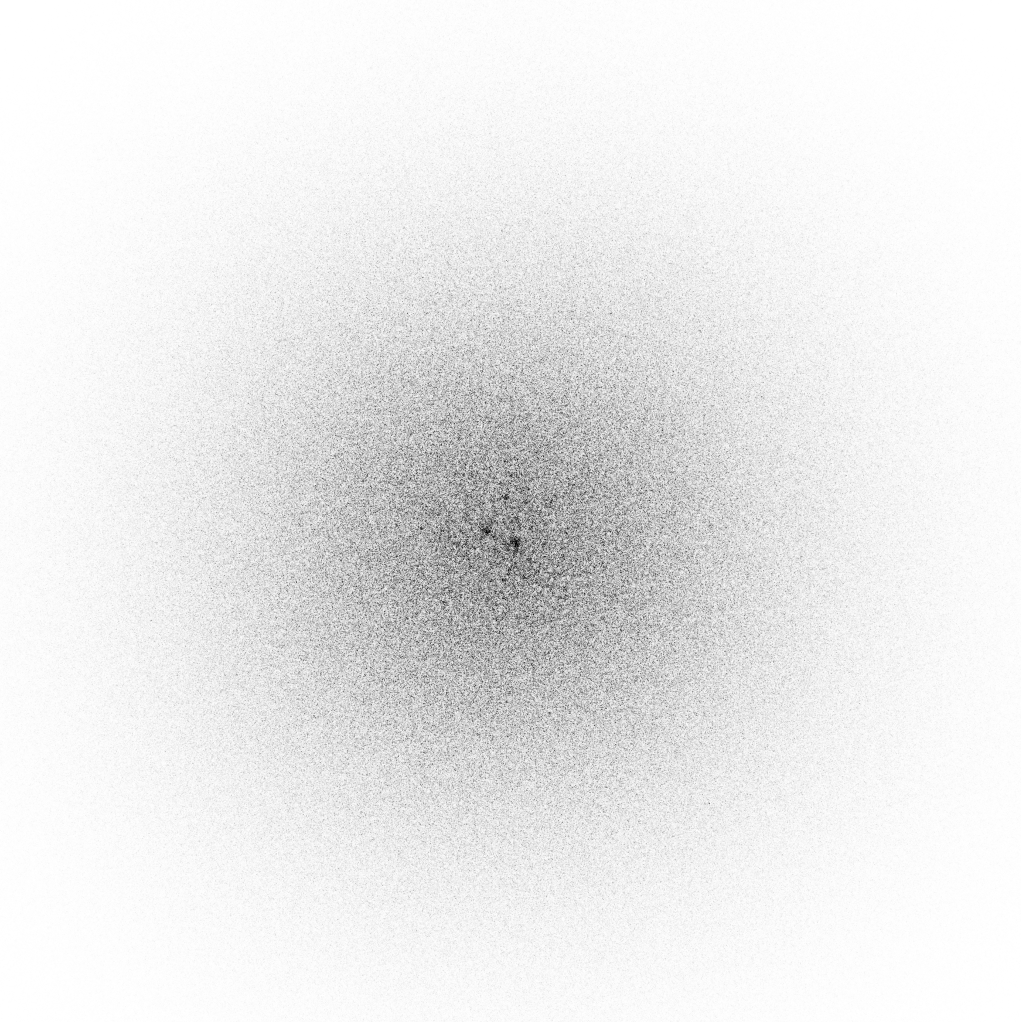} & \includegraphics[scale=0.4]{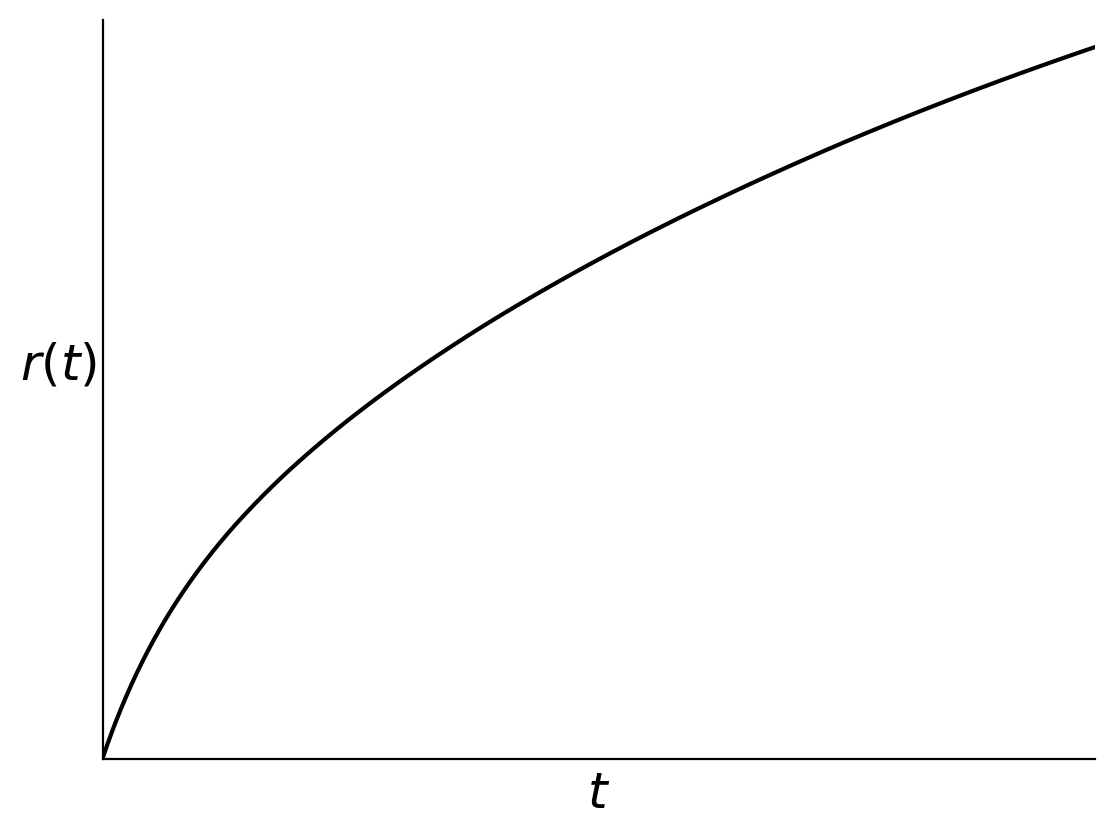}
\end{tabular}
\caption{The result of a simulation in $d=2$, $\lambda=0.1$ at time $T=2000$ with initial data $\psi_0=\ket{0}$.  On the left is the square of the wavefunction at time $T$, and on the right is the plot of $r(t)$ as a function of $t$.  Note that $|\psi_T|^2$ weakly resembles a Gaussian distribution, and that $r(t)$ appears to have square-root growth at long times.}
\label{fig:numerical}
\end{figure}

In these notes we provide a self-contained proof of the following result.
\begin{theorem}[Simplification of Theorem 1.1 in~\cite{black2025self}]
\label{thm:main-thm}
Let $\psi_0$ be the Kronecker delta at the origin and $\psi_t$ solve~\eqref{eq:RSE} on $\bbZ^d$, $d\geq 2$, and let $\kappa < \frac{1}{10}$.  Then for any
$\lambda^{-2}\leq T\leq \lambda^{-2-\kappa}$, the bound
\[
\frac{1}{T}\int_0^T \sum_{x\in\bbZ^d} |x|^2 |\psi_t(x)|^2\diff t \geq c \lambda^{-1} T^{1/2}
\]
holds with probability at least $1 - C\lambda^{1000}$.
\end{theorem}

Prior work towards understanding diffusion in the random Schrodinger equation can be characterized loosely into three threads.\footnote{Of course, this categorization is incomplete and overly simple.  Some very interesting works that the reader may also find
inspiring are the papers of Magnen, Poirot, and Rivasseau~\cite{magnen1997anderson,magnen1998ward,magnen1999renormalization}
and Duerinckx and Shirley~\cite{duerinckx2021new}.}

The first thread consists of \textit{direct perturbative} approaches.  This can be traced back to the work of van Hove~\cite{van1954quantum}, who first observed conditions under which perturbations $\lambda W$ to a Hamiltonian $H_0$ would have ``effective strength'' $\lambda^2$.  More specific to the Anderson model, Vollhardt and W\"olfle~\cite{vollhardt1980diagrammatic} described a diagrammatic approach which explains diffusive behavior and localization in $d\leq 2$.
The first mathematically rigorous work along these lines was by Herbert Spohn~\cite{Spohn} who considered the ``van Hove limit'' $\lambda\to 0$ and $t\sim \lambda^{-2}$ for the random Schrodinger equation, establishing a kinetic equation for the momentum distribution up to time $c\lambda^{-2}$.
This was improved by L\'aszl\'o Erd\H{o}s and HT Yau in~\cite{EYkinetic}, who reached time scales $C\lambda^{-2}$ and established a kinetic equation for the Wigner phase space distribution of $\psi_t$.  The difficulty with the ``direct perturbative'' approach is that to reach longer time scales one must go up to higher orders in perturbation theory and tame a combinatorial explosion of diagrams.  Nevertheless in a technical tour-de-force, Erd\H{o}s, Salmhofer, and Yau~\cite{ESYRSE,ESYAnderson,ESYrecollision} reached diffusive time scales on the order $\lambda^{-2-\eps}$ (with $\eps = \frac{1}{9800}$ on $\bbZ^3$, and $\eps=\frac{1}{370}$ on $\Real^3$), which was the first rigorous realization of some of the diagrammatic techniques of~\cite{vollhardt1980diagrammatic}.
This work was revisited in~\cite{hernandez2024quantum} which provided a different perspective on the diagrammatic expansions.  These papers only addressed convergence to a diffusive limit \textit{in expectation}, whereas convergence in higher order moments was established by Thomas Chen up to kinetic time in~\cite{Chen}.

The second thread consists of \textit{harmonic analysis} approaches to the random Schrodinger equation.  In~\cite{SSW}, Schlag, Shubin, and Wolff proved that in $d\leq 2$, eigenfunctions for the random Schrodinger operator have Fourier transforms localized to a $\lambda^2$-thick annulus around the corresponding level set of the dispersion relation. In dimension $2$, this proof used as input
the C\'ordoba $L^4$ argument used in two-dimensional restriction theory.  Later~\cite{Bourgain} established similar estimates without the use of wavepacket analysis, using instead stronger inputs from probability
Both of these papers pushed the first heuristic, that $e^{itH}$ behaves like $e^{it\Delta}$ up to time $\lambda^{-2}$, nearly as far as it could go by using the properties of the Laplacian in a strong
way.

The third thread of research is \textit{random matrix} analysis of related models.
The Anderson Hamiltonian $H=\Delta+\lambda V$
is itself a random matrix, but the fact that the randomness appears only on the diagonal has thus far precluded the use of traditional
ideas from the field.  Indeed, the most traditional random matrix models have randomness in \textit{every} entry.  Nevertheless, there has been recent and rapid progress on understanding random matrices with fewer and fewer random entries.  A simple example is the
random band matrix on $\bbZ^d_L := \bbZ^d/L\bbZ^d$, which has Hamiltonian $H$
\[
H_{xy} =
\begin{cases}
g_{xy}, &|x-y| \leq W \\
0, &\text{ else.}
\end{cases}
\]
The random entries $g_{xy}$ are independent up to the symmetry constraint $g_{xy}=g_{yx}$.  Recent progress on this model uses
the framework of computing moments for the resolvent $R(z)=(H-z)^{-1}$ via \textit{self-consistent equations}
developed first for Wigner matrices~\cite{erdos2009semicircle} (see also~\cite{erdos2013local} for a clear exposition)
and extended to random band matrices in~\cite{Teq}.
In a recent series of papers~\cite{Y3I,Y3II,Y3III}, it was shown that in $d\geq 8$ the eigenvectors of $H$ are completely delocalized so long as $W\geq L^{\eps}$ and $L$ is large enough.
Even more recently, this work has been extended to other random matrix models that are
closer to the Anderson model~\cite{Y2block} and, in spectacular recent breakthroughs,
even to dimension $d=2$~\cite{Dubova} and $d\geq 3$~\cite{dubova20253d}.
Until now, however, these random matrix methods have never been applied directly to the Anderson model.

In these notes we explain how to use ideas from each of these threads of prior work to prove Theorem~\ref{thm:main-thm}.
We break up the exposition into two sections.  Section 2 covers the required results up the kinetic time scale $\lambda^{-2}$
which were established in~\cite{black2025tail}.  Section 3 then explains how to prove results at the diffusive time scale and
proves Theorem~\ref{thm:main-thm}.   In these lecture notes the goal is to communicate my favorite ideas from the proofs, and to skip
details I find less interesting.  The reader can of course refer to~\cite{black2025tail,black2025self} for more details and stronger
results.

\vskip 12pt
\textbf{Acknowledgements.}
I thank the organizers of the ``PDE and Probability" summer school at Sorbonne Université for the invitation to deliver these lectures.  In particular, I would like to thank Antoine Gloria for his generous hospitality during this visit.  Of course the ideas and perspectives presented in these notes are the product of many discussions with Adam Black and Reuben Drogin, who also provided helpful comments and suggestions during the preparation of these lectures.  This work was supported by NSF award DMS-2303094.

\newpage
\section{The Kinetic Time Scale}
In this section we try to squeeze out as much information about $e^{itH}$ as possible just from considerations at short
times (short compared to the kinetic timescale $\lambda^{-2}$).  Previously it was known that one could obtain delocalization
lower bounds at scale $\lambda^{-2}$~\cite{SSW,Chen}.   In this chapter we establish this as a consequence of the fact that $e^{itH}\approx e^{it\Delta}$ for $t\ll\lambda^{-2}$, which itself has much farther reaching consequences.  In particular, we are able to prove
$\ell^p\to\ell^q$ bounds for the spectral projections of $H$ which we will use as a priori estimates to prove diffusion.

We can write $e^{itH}-e^{it\Delta}$ using the Duhamel formula as follows:
\begin{equation}
\label{eq:duhamel}
e^{-itH} - e^{-it\Delta} = i\lambda \int_0^t e^{-i(t-s)H} V e^{is\Delta} \diff s.
\end{equation}
Thinking of $V$ as a bounded potential, one naively has $e^{-itH}-e^{-it\Delta} = O(\lambda t)$.
In fact however there is a square-root cancellation in the integral above which leads to the estimate
\[
e^{-itH} - e^{-it\Delta} = O(\lambda \sqrt{t}).
\]
Such an estimate would imply that the effective strength of the potential is $\lambda^2$ rather than
$\lambda$, in the sense that for times $t\ll\lambda^{-2}$ the potential is not noticed.

So far I have been imprecise about the precise norm one should use to measure $e^{-itH} - e^{-it\Delta}$.
For example, one precise (and correct!) statement one can make is
\begin{equation}
\label{eq:moment-cancellation}
\sup_{\|\psi\|_{\ell^2}=1} \Big(\Expec \|e^{-itH}\psi - e^{-it\Delta}\psi\|_{\ell^2}^{2k}\Big)^{\frac{1}{2k}}
\lsim C_k\lambda \sqrt{t}.
\end{equation}
The bound~\eqref{eq:moment-cancellation} is useful in some applications, but it cannot directly be used
to say anything about \textit{eigenfunctions} of $H$, for the simple reason that eigenfunctions are
random functions and the statement above is about deterministic $\psi$.  A stronger and more useful bound
would swap the expectation and supremum above, and thus be a bound on operator norm as in
\begin{equation}
\label{eq:operator-norm}
\Big(\Expec \|e^{-itH}- e^{-it\Delta}\|_{\ell^2\to\ell^2}^{2k}\Big)^{\frac{1}{2k}}
\lsim C_k\lambda \sqrt{t}.
\end{equation}
Such a bound cannot be true on $\bbZ^d$ for the trivial reason that a random potential resembles any
deterministic potential on arbitrarily large sets (although extremely sparsely).  There are several
ways to remedy this.  One is to add a cutoff to the potential so that it has compact support, as done in~\cite{BDH}.   An alternate remedy, which we use in these notes, is to consider a model
on a discrete torus
$\bbZ^d_L := \bbZ^d / L\bbZ^d$ for a length $L = \lambda^{-10}$ (say), which is much larger than
any length scale relevant to the dynamics we consider.

We are now ready to state the main result of this chapter.
\begin{theorem}
\label{thm:propagator-kinetic}
Let $L=\lambda^{-100}$ and $H = \Delta_L + \lambda V$.
Then with probability at least $\exp(-cK^2)$, we have that the estimate
\[
\|e^{-itH} - e^{-it\Delta}\|_{\ell^2\to\ell^2} \leq K (\log\lambda^{-1})^2 \lambda \sqrt{t}.
\]
holds for all $t\in\Real$.
\end{theorem}

What we need in our application is a result on spectral projections.  To this end,
let $\chi\in C_c^\infty(\Real)$ be any smooth and compactly supported function.  We write $\chi(H)$
to mean the operator satisfying $\chi(H)\psi_E = \chi(E)\psi_E$ for any eigenfunction $H\psi_E = E\psi_E$
(since we are on $\bbZ^d_L$ anyway, $H$ is a finite dimensional matrix).  Then we define
$\chi_{\delta,E}(x) := \chi((x-E)/\delta)$.  Then we have the following result.
\begin{corollary}
\label{cor:projection-kinetic}
Let $L=\lambda^{-100}$ and $H = \Delta_L + \lambda V$.  Then with probability at least
$\exp(-cK^2)$ we have that
\[
\| \chi_{\delta,E}(H) - \chi_{\delta,E}(\Delta_L)\|_{\ell^2\to\ell^2}
\leq K (\log \lambda^{-1})^2 \lambda \delta^{-1/2}.
\]
\end{corollary}
The above result allows us to compare the spectral statistics of $\Delta_L$ with that of $H$ down
to intervals of width $\lambda^2$, and is the crucial ingredient in our approach to proving quantum diffusion.

Indeed, we use Corollary~\ref{cor:projection-kinetic} to establish the following less obvious consequence of Theorem~\ref{thm:propagator-kinetic}:
\begin{theorem}
\label{thm:resolvent-apriori}
Let $\eps>0$ be a small number and $N>0$ be a big number.  Suppose $E\in[-2d,2d]$ is not a critical value of the dispersion
relation $\omega$.
Then resolvent for $H =\Delta_L + \lambda V$ satisfies, for any $1\leq p\leq \frac65$ and $6\leq q\leq \infty$, the estimate
\[
\|R(E+i\eta)\|_{p\to q} \leq C_{\eps,N} \lambda^{-\eps} (\lambda^2\eta^{-1} + 1)
\]
with probability at least $1-C_{\eps,N} \lambda^N$.
\end{theorem}
To work with the above theorem more concisely we introduce the following notation.  We say that $B$ \textit{stochastically dominates} $A$, written $A\ldom B$, if for any $\eps,N>0$ there is a $C_{\eps, N}$ such that
\[
A \leq C_{\eps, N} \lambda^{-\eps} B
\]
holds with probability at least $1-C_{\eps,N} \lambda^N$.
We will derive Theorem~\ref{thm:resolvent-apriori} from Corollary~\ref{cor:projection-kinetic} later in Section~\ref{sec:apriori}.
The bulk of this chapter is dedicated to the proof of Theorem~\ref{thm:propagator-kinetic}, which we sketch below.

To prove Theorem~\ref{thm:propagator-kinetic} we iterate~\eqref{eq:duhamel} to write out $e^{-itH} - e^{-it\Delta}$ as a Dyson series,
\begin{equation}
e^{-itH} - e^{-it\Delta} = \sum_{j=1}^\infty e^{-it\Delta} (i\lambda)^j T_j(t),
\end{equation}
where
\[
T_1(t) := \int_0^t e^{-is\Delta} V e^{is\Delta} \diff s
\]
and for $j\geq 1$ we have
\[
T_j(t) := \int_{0\leq s_1\leq \cdots \leq s_j\leq t} V(s_j) V(s_{j-1})\cdots V(s_1)\diff \vec{s},
\]
with $V(s) := e^{-is\Delta} V e^{is\Delta}$.

A simple idea that would work to bound $\Expec \|T_j\|_{op}^k$ would be to use the moment method,
using
\[
\Expec \|T_j\|_{op}^{2k} \leq\Expec  \|T_j^*T_j\|^k_{op} \leq\Expec  \trace (T_j^*T_j)^k.
\]
The expectation on the right could be computed using a Wick expansion, involving the introduction of
a combinatorial explosion of terms.  Such an analysis was indeed done in~\cite{hernandez2024quantum},
but this is cumbersome and it is difficult to extract reasonable quantitative bounds.

To avoid any diagrammatic expansion we use two ideas.
The first is to observe that $T_1(t)$ is a random symmetric
matrix that is \textit{linear} in the randomness $V$.  This allows us to import a standard tool from
random matrix theory, the non-commutative Khintchine inequality.  The noncommutative Khintchine
inequality is introduced and proven in Section~\ref{sec:NCK}.
The second idea is a remarkable
reduction which allows us to bound the operator norms of $T_j$ in terms of the operator norm of $T_1$.
For the second trick to work, it is crucial that we work in operator norm since we make use of
an approximate structure of $T_2$ (in particular, we will use the inequality $\|AB\| \leq \|A\|\|B\|$).

\subsection{The non-commutative Khintchine inequality}
\label{sec:NCK}
The presentation in this section closely follows Section 3.1 of van Handel's ``Structured Random Matrices''~\cite{van2017structured}.

The main ingredient we need is a bound on the operator norm of random matrices $X$ of the form
\begin{equation}
\label{eq:random-X}
X = \sum_{j=1}^s g_j A_j,
\end{equation}
where $g_j$ are independent Gaussian variables and $A_j$ are symmetric random $n\times n$ matrices.
Note that any random matrix $X$ with jointly Gaussian entries can be written in this way.

As examples, note that a GOE random matrix (having variance $2N^{-1}$ on the diagonal and variance $N^{-1}$ on the off-diagonal)
can be written as
\begin{equation}
\label{eq:XGOE-def1}
X_{\rm GOE} := \frac{1}{\sqrt{N}}\sum_{i\leq j} g_{ij} E_{ij}
\end{equation}
where
\[
E_{ij} := \begin{cases}
\sqrt{2}\ket{i}\bra{i}, & i=j \\
\ket{i}\bra{j}+\ket{j}\bra{i}, & i\not=j.
\end{cases}.
\]
Another example is a diagonal random matrix,
\begin{equation}
\label{eq:Xdiag-def}
X_{\rm diag} := \sum_{i} g_i \ket{i}\bra{i}.
\end{equation}

In both cases, one can verify that $\|X\|_{\rm op} \lsim \|\sum_j A_j^2\|^{1/2}$ with high probability, which
is a noncommutative version of the square root cancellation expected if $A_j$ were scalars.
\begin{theorem}[Non-commutative Khintchine inequality]
\label{thm:nck}
For matrices $X$ of the form~\eqref{eq:random-X},
\begin{equation}
(\Expec \trace[X^{2p}])^{1/2p} \leq \sqrt{2p-1} \Big(\trace[(\Expec X^2)^p]\Big)^{1/2p}
\end{equation}
\end{theorem}
This result is due to Lust-Piquard and Pisier~\cite{lust1986inegalites,lust1991non}.

Before we prove Theorem~\ref{thm:nck} we note that it implies a concentration inequality on the operator norm of $X$, up to a logarithmic loss in the dimension.
\begin{corollary}
\label{cor:nck-op-bd}
Let $X$ be of the form~\eqref{eq:random-X} and let $n$ be the dimension of $X$ (so $X$ is an $n\times n$ matrix).  Then
\begin{equation}
\label{eq:expec-op}
\Expec \|X\|_{op} \lsim \sqrt{\log (n)} \|\sum_{j=1}^s A_j^2\|_{op}^{1/2}.
\end{equation}
Moreover, for $\alpha > 4$ we have the bound
\begin{equation}
\label{eq:conc-op}
\Prob(\|X\|_{op} \geq \alpha \sqrt{\log(n)} \|\sum_{j=1}^s A_j^2\|_{op}^{1/2}) \leq
\exp(- \frac{\log 2}{8} \alpha^2 \log n).
\end{equation}
\end{corollary}

To see where the factor of $\sqrt{\log{n}}$ comes from, observe that with $X_{\rm diag}$ defined as in~\eqref{eq:Xdiag-def},
\[
\|X_{\rm diag}\|_{op} = \sup_{1\leq i\leq n} |g_i|.
\]
This is on the order $\sqrt{\log n}$ with high probability.  Note that in this case $\sum_i A_i^2 = \Id$.
On the other hand, for the GOE ensemble~\eqref{eq:XGOE-def1},
\[
N^{-1} \sum_{ij} E_{ij}^2 = \frac{N+1}{N} \Id,
\]
so that the non-commutative Khintchine inequality implies
\[
\|X_{\rm GOE}\|_{op} \lsim \sqrt{\log N}.
\]
This is lossy, as the truth is that the largest eigenvalue of $X_{\rm GOE}$ is close to $2$ with high probability.   In general,
determining the precise dependence on dimension is an interesting problem, but it is completely irrelevant to our analysis which is insensitive to logarithms.

\noindent
\textbf{Exercise. } Derive Corollary~\ref{cor:nck-op-bd} from Lemma~\ref{thm:nck}.
\Sol{
\begin{proof}
For an $n\times n$ symmetric matrix $X$, we have
\begin{equation}
\|X\|_{op}^{2p} \leq \trace[ X^{2p}] = \sum_{j=1}^n \sigma_j^{2p} \leq n \|X\|_{op}^{2p}.
\end{equation}
Taking $2n$-th roots we obtain $\|X\|_{op} \simeq \trace[X^{2p}]^{1/2p}$ for $p \gtrsim \log n$.
Applying Theorem~\ref{thm:nck} directly with $p=\log n$ we obtain~\eqref{eq:expec-op}.

Now we prove the concentration inequality~\eqref{eq:conc-op}.  For simplicity we set
$B = \sum_{j=1}^s A_j^2$.
To get the concentration inequality we use $\|X\|_{op}^{2k} \leq \trace X^{2k}$
and apply Theorem~\ref{thm:nck} as follows:
\begin{equation}
\begin{split}
\Prob(\|X\|_{op} \geq \alpha \sqrt{\log n} \|B\|_{op})
&\leq \alpha^{-2k} (\log n)^{-k}  \|B\|_{op}^{-k}
\Expec(\|X\|_{op}^{2k}) \\
&\leq \alpha^{-2k} (\log n)^{-k}  \|B\|_{op}^{-k}
\Expec(\trace[X^{2k}]) \\
&\leq \alpha^{-2k} (\log n)^{-k}  \|B\|_{op}^{-k}
(2k)^k \trace[B^k] \\
&\leq \Big(\frac{2k n^{1/k}}{\alpha^2 \log n}\Big)^k.
\end{split}
\end{equation}
Now take $k = \frac{1}{8}\alpha^2 \log n$.  For $\alpha > 4$, $k > 2\log n$ so $n^{1/k} < 2$.
In particular $2kn^{1/k} < 4k = \frac{1}{2}\alpha^2 \log n$ so the fraction in brackets is bounded by $\frac{1}{2}$.
\end{proof}
}

\subsubsection{Proof of the non-commutative Khintchine inequality}

The key point of the proof of Theorem~\ref{thm:nck} is the following bound, which can be interpreted as saying that the
worst case is that $A_j$ all commute.
\begin{lemma}
\label{lem:trace-ineq}
For any symmetric matrices $A$ and $B$, and $0\leq \ell\leq 2p-2$
\begin{equation}
\trace[A B^{2p-2-\ell} A B^\ell] \leq \trace[A^2 B^{2p-2}].
\end{equation}
\end{lemma}
\begin{proof}
We can always rotate into a basis so that $B$ is diagonal with entries $B_{jk} = b_j \delta_{jk}$.  Then the trace
can be estimated using Holder's inequality as follows:
\begin{equation}
\begin{split}
\trace[A B^{2p-2-\ell} A B^\ell] &=
\sum_{j,j'} b_j^\ell b_{j'}^{2p-2-\ell} |a_{j,j'}|^2 \\
&\leq \Big(\sum_{j,j'} |b_j|^{2p-2} |a_{j,j'}|^2\Big)^{\ell/(2p-2)}
\Big(\sum_{j,j'} |b_{j'}|^{2p-2} |a_{j,j'}|^2\Big)^{(2p-2-\ell)/(2p-2)}.
\end{split}
\end{equation}
Now we recognize each sum in the product above as being equal to $\trace[A B^{2p-2} A] = \trace[A^2 B^{2p-2}]$.
\end{proof}

\begin{lemma}[Matrix Jensen's inequality]
Let $A$ be a symmetric matrix and let $\varphi:\Real\to\Real$ be a convex function.  Then
\begin{equation}
\sum_{j=1}^d \varphi(a_{jj})  \leq \trace[\varphi(A)].
\end{equation}
\end{lemma}
\begin{proof}
Write $A = Q \Lambda Q^*$ where $\Lambda$ is diagonal, and $q_{jk}$ are the entries of $Q$.  Then
\[
a_{jj} = \sum_k \lambda_k |q_{jk}|^2.
\]
Then by Jensen's inequality (using that $\sum_k |q_{jk}|^2 = 1$ for all $j$,
\begin{equation}
\begin{split}
\sum_j \varphi(a_{jj})
&= \sum_j \varphi\Big( \sum_k \lambda_k |q_{jk}|^2\Big) \\
&\leq \sum_j  \sum_k \varphi(\lambda_k) |q_{jk}|^2 \\
&= \sum_k \varphi(\lambda_k) = \trace[\varphi(A)].
\end{split}
\end{equation}
\end{proof}

\begin{lemma}[Matrix Holder's inequality]
\label{lem:holders-ineq}
For symmetric matrices $A$ and $B$ and $\frac{1}{p} + \frac{1}{p'} = 1$ with $p>1$,
\begin{equation}
\trace[AB] \leq \trace[|A|^p]^{1/p} \trace[|B|^{p'}]^{1/p'}.
\end{equation}
\end{lemma}
\begin{proof}
Without loss of generality we can take $B$ to be diagonal.  Then
\begin{equation}
\trace[AB] = \sum_{j} b_{jj} a_{jj}
\end{equation}
The result now follows from Holder's inequality, since $\sum |a_{jj}|^p \leq \trace[|A|^p]$.
\end{proof}

\begin{proof}[Proof of Theorem~\ref{thm:nck}]
We use the Gaussian integration by parts formula $\Expec gf(g) = \Expec f'(g)$ for Gaussians $g$ to compute
\begin{equation}
\begin{split}
\Expec[\trace[X^{2p}]]
&= \sum_{j=1}^s \Expec[ g_j \trace[A_j X^{2p-1}]] \\
&= \sum_{\ell=0}^{2p-2} \sum_{j=1}^s \Expec[\trace[A_j X^{2p-2-\ell}A_jX^{\ell}]] \\
&\leq (2p-1) \Expec [\trace[(\sum_{j=1}^s A_j^2) X^{2p-2}]] \\
&\leq (2p-1) \trace[ (\Expec X^2)^p]^{1/p} \trace[\Expec X^{2p}]^{1-\frac1p}.
\end{split}
\end{equation}
In the last step we have used the matrix Holder's inequality and the identity $\sum_{j=1}^s A_j^2 = \Expec X^2$.  The
proof now follows from rearranging.
\end{proof}

Let us take a moment to reflect on the proof of the non-commutative Khintchine inequality.  The identity
\[
\Expec \trace[X^{2p}] = \sum_{\ell=0}^{2p-2} \sum_j \Expec \trace[A_j X^{2p-2-\ell} A_j X^{\ell}]
\]
can be interpreted diagrammatically.  Indeed, the Wick expansion dictates that the left hand side can be expanded as a sum over the
$(2p-1)!!$ perfect matchings of $[2p]$.  The right hand side can be thought of as a decomposition of these matchings according to
where the first matching got sent.  The inequality
\[
\trace[A_j X^{2p-2-\ell} A_j X^{\ell}]
\leq
\trace[A_j^2 X^{2p-2}],
\]
combined with the Holder inequality, allows us to argue that the worst case matching is the one in which each
matrix factor gets paired with its neighbor.  In particular it allows us to ``uncross'' all of the diagrams.
This inequality is in fact an equality when all of the $A_j$ commute (indeed, in this case
$A_jX^a A_j X^b = A_j^2 X^{a+b}$).  On the other hand in the GOE example, the $E_{ij}$ matrices are very much non-commutative,
so the inequality is very lossy.

\subsubsection{Applying non-commutative Khintchine to $T_1(t)$}
Now we turn back to the random Schrodinger equation on $\bbZ^d_L$.  In this section we
write $\Delta = \Delta_L = \Delta_{\bbZ^d_L}$.

Recall that we need a bound on the first collision operator,
\begin{equation}
T_1(t) = \int_0^t e^{is\Delta} V e^{-is\Delta} \diff s,
\end{equation}
in operator norm.  Note that $T_1(t)$ can be written in the form
\[
T_1(t) = \sum_{j\in\bbZ^d_L} g_j A_j(t),
\]
where
\[
A_j(t) = \int_0^t e^{is\Delta} \ket{j}\bra{j} e^{-is\Delta}\diff s,
\]
and $\ket{j}\bra{j}$ is the rank-one projection onto the site at $j\in\bbZ^d_L$.

We have set things up so that we can direcly apply the non-commutative Khintchine inequality directly
to $T_1(t)$:
\begin{lemma}
\label{lem:simple-Q-bd}
With $L=\lambda^{-100}$, operator $T_1(t)$ satisfies the estimates
\begin{equation}
\Expec( \|T_1(t)\|_{op} ) \lsim
\begin{cases}
(\log \lambda^{-1}) \sqrt{t \log t}, & d=2 \\
(\log \lambda^{-1}) \sqrt{t}, & d\geq 3.
\end{cases}
\end{equation}
for $t\ll R$.
\end{lemma}

We have made two simplifications in the statement of the above result.  The first is that we did not
state the tail bounds for $\|T_1(t)\|$, but this follows from just applying the correct version of the
non-commutative Khintchine inequality (with the tail bound, not just the point estimate).  The second
is that we have obtained a bound only for a single (arbitrary) time $t$.  It is possible to improve this
to a bound on all $t\ll \lambda^{-10}$ by a union bound argument.

Now we prove Lemma~\ref{lem:simple-Q-bd} by a more or less direct computation.
\begin{proof}
By the non-commutative Khintchine inequality, it suffices to bound the operator norm of
\begin{equation}
\label{eq:B-int}
B(t) := \sum_j A_j(t)^2 = \sum_j \int_0^t \int_0^t e^{is'\Delta} \ket{j}\bra{j}
e^{i(s-s')\Delta} \ket{j}\bra{j} e^{-is\Delta}\diff s \diff s'.
\end{equation}
Using that $\sum_j \ket{j}\bra{j} = \Id$ and that $\braket{j | e^{i\sigma \Delta} |j }
= \braket{0 | e^{i\sigma\Delta}|0} = f_d(t)$, we can write
\[
B(t) = \int_0^t \int_0^t f_d(s-s') e^{i(s'-s)\Delta}\diff s\diff s',
\]
so that
\[
\|B(t)\|_{op} \leq \int_0^t \int_0^t |f_d(s-s')|\diff s\diff s' \leq t \int_{-t}^t |f_d(\sigma)|\diff \sigma.
\]
A Fourier-analysis computation shows that
$|\braket{0 | e^{i\tau \Delta_{\bbZ^d}} | 0}|\lsim (1+|\tau|)^{-d/2}$.  On $d=2$ this is borderline
integrable and we pick up a factor of $\log t$, and otherwise this is bounded directly by $t$.  The
proof is now completed with an application of NCK.
\end{proof}

\Exercise{Show that $\|T_1(t+s)\|_{op} \leq \|T_1(t)\|_{op}+\|T_1(s)\|_{op}$.  Deduce that
for $T\geq 1$,
\[
\sup_{t\in[0,T]}\|T_1(t)\|_{op} \leq
\sup_{t\in[0,1]}\|T_1(t)\|_{op} + \sum_{j=1}^{\lceil \log_2 T\rceil} \|T_1(2^j)\|_{op}.
\]
Using this and a union bound one may turn an estimate for $\|T_1(t)\|_{op}$ at a single time into an estimate for $\sup_{t\in[0,T]} \|T_1(t)\|_{op}$.
}

\subsection{Bounds on $T_k(t)$ from $T_1(t)$}
\label{sec:Tkarg}

Now we use the bounds on $T_1$ to prove bounds on the remaining terms in the Duhamel expansion.  By iterating the Duhamel formula we obtain the identity
\begin{equation}
\label{eq:duhamel-exp}
e^{-itH} = e^{-itH_0} + \sum_{j=1}^k (i\beta)^j T_j(t) + i\beta \int_0^t e^{-i(t-s)H} V T_k(s)\diff s,
\end{equation}
where $T_j$ is the operator defined by $T_0(t) = e^{-itH_0}$ and recursively using
\begin{equation*}
T_j(t) = \int_0^t e^{-i(t-s)H_0} V T_{j-1}(s)\diff s.
\end{equation*}

For simplicity we will see how to deal with $T_2(t)$, which is given by
\[
T_2(t) = \int_{0\leq s_1\leq s_2\leq t} V(s_2)V(s_1)\diff s_1\diff s_2.
\]
We can think of $T_2(t)$ as the integral over a triangular region in $\Real^2$ of the operators
$V(s_2)V(s_1)$, as in Figure~\ref{fig:T2}.

\begin{figure}
\centering
\begin{tikzpicture}
\draw[->] (-0.2, 0) -- (4.5, 0) node[right] {};
\draw[->] (0, -0.2) -- (0, 4.5) node[above] {};

\node at (2, -0.3) {$\frac{t}{2}$}; \node at (-0.2,2) {$\frac{t}{2}$};
\node at (4.3, -0.2) {$s_1$}; \node at (-0.2,4.2) {$s_2$};
\fill[fill=blue!30] (0,2) rectangle (2,4);
\fill[orange!40] (0,0) -- (2,2) -- (0,2) -- cycle;
\fill[orange!40] (2,2) -- (4,4) -- (2,4) -- cycle;
\draw[dashed] (0,2) -- (4,2);
\draw[dashed] (2,0) -- (2,4.0);
\draw[thick] (0,0) -- (4,4) -- (0,4) -- cycle;
\end{tikzpicture}
\caption{A schematic for the decomposition of the operator $T_2(t)$.  The region in the blue square corresponds to a term of the form $T_1(t/2)^2$.}
\label{fig:T2}
\end{figure}
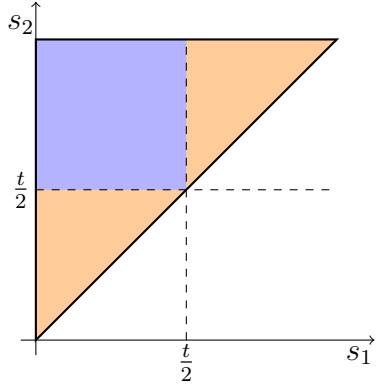
We can subdivide this triangle into a square of side
length $t/2$ and two smaller triangles.  The integral over the square corresponds to the operator
\begin{align*}
\int_{t/2}^t \int_{0}^{t/2} V(s_2) V(s_1)\diff s_1 \diff s_2
&= \Big( \int_{t/2}^t V(s_2)\diff s_2 \Big) \Big(\int_0^{t/2} V(s_1)\diff s_1\Big) \\
&= e^{-it\Delta/2} T_1(t/2) e^{it\Delta/2} T_1(t/2).
\end{align*}
The triangle in the lower left integrates to the operator $T_2(t/2)$ and the triangle in the upper
right integrates to $e^{-it\Delta/2} T_2(t/2)e^{it\Delta/2}$.  Since $e^{it\Delta}$ is unitary
we have the following bound in operator norm:
\[
\|T_2(t)\|_{op} \leq \|T_1(t/2)\|_{op}^2 + 2 \|T_2(t/2)\|_{op}.
\]
Iterating this bound yields
\[
\|T_2(t)\|_{op} \leq 2t \sup_{s\in [0,1]}\|T_2(s)\|_{op}
+ \sum_{j=1}^{\lfloor \log_2 t\rfloor} 2^{j-1} \|T_1(t/2^j)\|_{op}^2.
\]
Each of the terms in the sum is of order $\lessapprox t$, so that we conclude
\[
\|T_2(t)\|_{op} \lessapprox t,
\]
where $\lessapprox$ is hiding a logarithmic factor in $t$.

In the exercises below we show how to extend this to higher order terms $T_j$.

\Exercise{Prove the decomposition formula
\begin{equation}
\label{eq:Tkdecomp}
T_k(s+t) = \sum_{j=0}^k T_j(s)T_{k-j}(t).
\end{equation}}

\Exercise{Using the decomposition formula above, observe the bound
\begin{equation}
\|T_k(t)\|_{op} \leq \sum_{\substack{\vec{m} \in \bbN_{\geq 0}^k \\ \sum m_j = k}}
\prod_{j=1}^k \|T_{m_j}(t/k)\|_{op}.
\end{equation}
Assume that
\[
\|T_j(t)\|_{op} \leq (C_j A^2 (t\log t))^{k/2}
\]
for $j<k$, and let $\bar{C} := \max_{j<k} C_j$.  Using the above bound, prove that
\[
\|T_k(t)\|_{op} \leq (4k^{-1/2} A \bar{C} (t\log t)^{1/2})^k.
\]
}

\Exercise{Using the above exercise, conclude that
\[
\|T_k(t)\|_{op} \leq (C k^{-1} A^2 (t\log t))^{k/2}
\]
with probability at least $\exp(-cA^2)$.
}

By summing the above estimate for $T_k(t)$, we obtain a result for $e^{-itH}$.
\begin{theorem}
\label{thm:Tkbds}
The family of estimates
\begin{equation}
\label{eq:Tkopbd}
\|T_k(t)\|_{op} \leq (C \alpha k^{-1} \log(R)(t\log t))^{k/2}
\end{equation}
holds with probability at least $\exp(-c\alpha^2)$.
In particular, for $t \ll (\log \beta^{-1})^{-1} \beta^{-2}$, it follows that
\begin{equation}
\label{eq:op-compare}
\|e^{-itH} - e^{-itH_0}\|_{op} \lsim \beta |\log\beta| \sqrt{t}.
\end{equation}
\end{theorem}

\subsection{Bounds for spectral projections}
\label{sec:apriori}
To conclude the chapter we show how to derive Theorem~\ref{thm:resolvent-apriori} from Corollary~\ref{cor:projection-kinetic}.
First, we list some easy consequences of Corollary~\ref{cor:projection-kinetic} (which we leave to the reader to state
precisely and prove if desired):
\begin{itemize}
\item On $\bbZ^d/L\bbZ^d$ with $L=\lambda^{-100}$, the spectrum of $H$ is very likely to be contained in $[-2d-\eps,2d+\eps]$ for $\eps = (\log\lambda^{-1})^C\lambda^2$.
\item Any eigenfunction $H\psi = E\psi$ has its Fourier transform localized (in an $\ell^2$ sense) to the set
$\{|\omega(\xi)-E|\lsim \lambda^2\}$.
\item As a consequence of the above item, if $E$ is not a critical value of $\omega$ and $H\psi=E\psi$ then $\psi$ cannot have too
much of its $\ell^2$ mass in any ball of radius $r\ll (\log\lambda^{-1})^{-C}\lambda^{-2}$.
\end{itemize}

What we need to reach diffusive timescales (and length scales) are $\ell^p\to\ell^q$ mapping properties of the resolvent.
At the moment what we have are estimates of the form
\[
\|\Pi_{\delta,E}(H) - \Pi_{\delta,E}(\Delta)\|_{2\to 2} \lessapprox \lambda \delta^{-1/2}.
\]
It may be surprising at first that we can say anything about $\|R(z)\|_{p\to q}$ using only this bound.  The first observation
is that we can transfer bounds when $p=2$ or $q=2$.  In fact, we have the following result:
\begin{lemma}
For $z=E+i\eta$ with $\eta \gtrsim \lambda^{-2}$,
\[
\|R^{1/2}\|_{2\to X} \ldom \|R_0^{1/2}\|_{2\to X}
\]
\end{lemma}
\begin{proof}[Proof sketch]
We write
\[
R^{1/2} = R_0^{1/2} (\Delta-z)^{1/2} R^{1/2}.
\]
So then it suffices to prove that
\[
\|(\Delta-z)^{1/2}R^{1/2}\|_{2\to 2} \ldom  1.
\]
We write out a partition of unity
\[
1 = \sum_j \chi_j
\]
where $\chi_0$ is a smooth cutoff to the interval $[E-\eta,E+\eta]$, and for $j\geq 1$
$\chi_j$ is supported on $|t-E|\sim 2^j\eta$.  Write $\Pi^0_j := \chi_j(\Delta)$ and
$\Pi_j := \chi_j(H)$.  Then
\[
(\Delta-z)^{1/2} R^{1/2} = \sum_{k,\ell} (\Delta-z)^{1/2} \Pi^0_k \Pi_\ell R^{1/2}.
\]
The diagonal terms contribute a logarithmic factor:
\[
\|\sum_k (\Delta-z)^{1/2}\Pi^0_k \Pi_kR^{1/2}\|
\lsim \sum_k 2^{k/2} \eta^{1/2} \eta^{-1/2} 2^{-k/2}\lsim \sum_k 1 \lsim \log\eta^{-1}.
\]
For the off-diagonal terms we can use Corollary~\ref{cor:projection-kinetic}.  For example, for the terms $k > \ell$ we write
\[
\Pi^0_k \Pi_\ell = \Pi_k \Pi_\ell + (\Pi^0_k-\Pi_k)\Pi_\ell = (\Pi^0_k-\Pi_k) \Pi_\ell,
\]
so that
\[
\|\Pi^0_k\Pi_\ell\|_{2\to 2} \leq \|\Pi^0_k-\Pi_k\| \ldom \lambda 2^{-k/2}\eta^{-1/2}.
\]
Therefore,
\[
\|(\Delta-z)^{1/2} \Pi_0^k \Pi_\ell R^{1/2}\|_{2\to 2}
\lsim (2^{k/2}\eta^{1/2}) (\lambda 2^{-k/2}\eta^{-1/2}) (2^{-\ell/2} \eta^{-1/2}
\lsim \lambda 2^{-\ell/2} \eta^{-1/2},
\]
which is a convergent sum in $\ell$.
\end{proof}

Now we explain how to use this.  Of course the above lemma also implies
\[
\|R^{1/2}\|_{X\to 2} \ldom \|R_0^{1/2}\|_{X\to 2}
\]
by duality.  Therefore for any $p\leq 2\leq q$ and $R=R(E+i\eta)$ and $\eta \gtrsim\lambda^2$,
\[
\|R\|_{p\to q} \ldom \|R_0^{1/2}\|_{p\to 2}\|R_0^{1/2}\|_{2\to q}.
\]
Moreover, if $z_0 = E+i\lambda^2$ and $z=E+i\eta$ for $\eta \ll \lambda^2$, we can write
\begin{equation}
\label{eq:small-eta-R}
R(z) = R^{1/2}(z_0) (R^{-1/2}(z_0) R(z) R^{-1/2}(z_0)) R^{1/2}(z_0).
\end{equation}
The operator in the brackets in the middle satisfies
\[
\|R^{-1/2}(z_0) R(z) R^{-1/2}(z_0)\|_{2\to 2} \leq \lambda^2\eta^{-1},
\]
so that for any $\eta>0$ we have the estimate
\[
\|R(z)\|_{p\to q} \ldom (\lambda^2\eta^{-1} + 1)\|R_0(z_0)\|_{p\to 2} \|R_0(z_0)\|_{2\to q}.
\]

\subsubsection{Resolvent bounds for $\Delta$ and $H$}
At this point it is clear that we need estimates for $\|R_0^{1/2}(E+i\eta)\|_{2\to q}$.  We can first decompose $R_0^{1/2}$ into
spectral projections,
\[
R_0^{1/2} = \sum_j R_0^{1/2} \chi_j(\Delta),
\]
where $\chi_j$ on a window of width $2^j\delta$.  If we had an estimate of the form
\[
\|\chi_j(\Delta)\|_{2\to q} \lsim 2^{-j/2} \eta^{1/2},
\]
then by summing over $j$ it would follow that
\[
R_0^{1/2} \lsim \log \eta^{-1}.
\]
Now $\chi_j(\Delta)$ can be explicitly written as a Fourier multiplier:
\[
\Ft{\chi_j(\Delta) f} (\xi) = \chi_j(\omega(\xi)) \Ft{f}(\xi).
\]

For $\chi_j$ we have the following estimate, which is essentially equivalent to the Tomas-Stein restriction estimate~\cite{TomasStein},
and we repeat the proof below.
\begin{lemma}[Tomas-Stein estimate]
Let $\chi$ be a bounded function supported in an interval $I = [E_0-\eps,E_0+\eps]$ not containing any critical values of
$\omega$.  Then on $d\geq 2$ we have the estimate
\[
\|\chi(\Delta)\|_{2\to 6} \lsim \eps^{1/2} \log\eps^{-1}.
\]
\end{lemma}

As a consequence of this lemma and the ``transfer'' argument, we have established Theorem~\ref{thm:resolvent-apriori}, which
we restate below.
\begin{theorem}
\label{thm:resolvent-apriori2}
The resolvent for $H =\Delta_L + \lambda V$ satisfies, for any $1\leq p\leq \frac65$, $6\leq q\leq \infty$,
\[
\|R\|_{p\to q} \ldom \lambda^2\eta^{-1} + 1.
\]
\end{theorem}

\begin{proof}
By a duality argument it suffices to prove that
\[
\|\chi(\Delta)^2\|_{\frac65\to 6} \lsim \eps\log\eps^{-1}
\]
We drop the square, just relabeling $\chi\mapsto\chi^2$.  We use a decomposition of $\chi$ as follows:
\[
\chi = \sum_k \phi_k,
\]
where $\phi_k$ is smooth to scale $2^{-k}$ and supported in an interval of width $2^{-k}$ around $E_0$.  There are
about $\log\eps^{-1}$ terms in this decomposition, going from unit scale to scale $\eps$.   We have
$\|\phi_k\|_\infty \lsim \eps 2^{-k}$, so that
\[
\|\phi_k(\Delta)\|_{2\to 2} \lsim 2^{k} \eps.
\]
To prove that $\|\phi_k(\Delta)\|_{\frac65\to 6}\lsim 1$ it therefore suffices by interpolation to prove that
\[
\|\phi_k(\Delta)\|_{1\to \infty} \lsim 2^{-\frac12 k} \eps.
\]
The operator $\phi_k(\Delta)$ is a convolution operator with kernel
\[
K_k(x) = \int_{\Torus^d} e^{i\xi x} \phi_k(\omega(\xi))\diff \xi.
\]
Since the function $\phi_k(\omega(\xi))$ oscillates (and is mean-zero) on scales $2^{-k}$, and is also smooth to scale $2^{-k}$,
we have that $|K_k(x)|$ is quite small unless $|x|\sim 2^k$.  For such $x$, one can decompose the integral above using the coarea
formula along the slices $X_E := \{\omega(\xi) = E\}$:
\[
K_k(x) = \int_{[E-2^{-k},E+2^{-k}]} \phi_k(E)
\Big(\int_{X_E} e^{i\xi x} \frac{1}{|\nabla\omega(\xi)|} \diff \mcal{H}^{d-1}(\xi)\Big) \diff E.
\]
We treat the term $|\nabla \omega|^{-1}$ as just a smooth weight because we assume that $E$ is away from critical values of $\omega$.
One can now decompose the surface $X_E$ into coordinate charts and apply a stationary phase estimate.
The fact that there is at least one direction of nonvanishing curvature at every point of $X_E$ implies that
\[
\Big|\int_{X_E} e^{i\xi x} \frac{1}{|\nabla\omega(\xi)|} \diff \mcal{H}^{d-1}(\xi)\Big|
\lsim |x|^{-1/2}.
\]
Therefore,
\[
|K_k(x)| \lsim 2^{-k/2}\eps
\]
as desired, and this completes the proof.
\end{proof}

As a consequence, we obtain that
\[
\|R_0^{1/2}(E+i\eta)\|_{2\to 6} \lsim \log\eta^{-1}.
\]

\subsubsection{$\ell^p$ bounds for eigenfunctions}
An interesting consequence of the resolvent bounds is a bound on the $\ell^p$ norm of eigenfunctions of $H$ on $\bbZ^d_L$.  Indeed,
let $\psi_k$ satisfy $H\psi_k=E_k\psi_k$ and $\|\psi_k\|_2=1$, so that it is an $\ell^2$-normalized eigenfunction of $H$ with energy $E_k$.  Then
\[
\psi_k
= \lambda R^{1/2}(E_k+i\lambda^2) \psi_k
\]
so that
\[
\|\psi_k\|_{\ell^6} \leq \lambda \|R^{1/2}(E_k+i\lambda^2)\|_{2\to 6} \lsim \lambda.
\]
This implies that on $\bbZ^2$ most eigenfunctions are of order $\lambda$ in $\ell^6$ norm.    In higher dimensions one can
obtain bounds for $\ell^p$ with $p_d>2$ but $p_d\to 2$ as $d\to\infty$.

\newpage
\section{The Diffusive Time Scale}
In this chapter of the lecture notes, we analyze timescales $t\gg \lambda^{-2}$ on which the potential now has a strong effect.  Before we try to compute, it is instructive to understand what
one expects to be true.  For this purpose it is helpful to return to the setting of $\Real^d$.  For simplicity we consider a stationary Gaussian potential.  Such a potential is completely characterized by its two-point correlation function,
\[
K(x) := \Expec V(0)V(x).
\]
To observe diffusion in the random Schrodinger equation
\[
i\partial_t\psi = \Delta\psi + \lambda V \psi,
\]
one looks at the evolution of observables $\braket{\psi_t, \Op(a_0)\psi_t}$, where $a_0\in C^\infty(\Real^d_x\times\Real^d_\xi)$ is a smooth function on phase space and $\Op(a_0)\in\mcal{B}(L^2(\Real^d))$
is the quantization of it (for example, using the Weyl quantization).  One can derive heuristically (see for example~\cite{Spohn}) the following effective kinetic equation for the phase space density $a_t(x,p)$:
\[
\partial_t a_t(x,\xi) + \xi\cdot\nabla_x a_t(x,\xi) = \lambda^2\int_{\Real^d} \delta(|\xi|^2/2-|\xi'|^2/2) \Ft{K}(\xi-\xi') [ a(x,\xi')-a(x,\xi)]\diff \xi'.
\]
The term on the left hand side describes free transport by particles along lines (solving $\dot{x}=\xi$ and $\dot{\xi}=0$).  The right hand side is a collision term that scatters particles at
momentum $\xi'$ to momentum $\xi$ (and vice-versa).  The $\lambda^2$ gives the effective strength of the potential (corresponding to the fact that, up to time $\lambda^{-2}$, the potential
is not noticed), the $\delta(|\xi|^2/2-|\xi'|^2/2)$ term enforces conservation of kinetic energy, and $\Ft{K}(\xi-\xi')$ is the scattering kernel corresponding to the potential $V$.
This limit was established first in the spatially homogeneous case and for small kinetic times by Spohn~\cite{Spohn}, then for the full linear Boltzmann equation and at arbitrary kinetic times
in~\cite{EYkinetic}, and finally for diffusive time scales in~\cite{ESYRSE}.

Note that
for $V$ smooth, $\Ft{K}$ is localized so that each scattering event can only deflect a particle by a small angle.
In contrast, on $\bbZ^d$ one has $K(x) = \delta_0(x)$ where $\delta_0$ is a Kronecker delta at the origin so that $\Ft{K}(\xi)=1$.  Moreover, $|\xi|^2/2$ is replacd by the dispersion relation $\omega(\xi)$.
Therefore, after a scattering event the momentum should be uniformly distributed on a kinetic energy shell $\{\omega(\xi)=E\}$ (weighted by $|\nabla\omega|^{-1}$ to account for the small thickness of the shell).   Thus the position of the particle at time $T$ is (classically) described by a random walk of $\sim \lambda^2T$ steps of size $\lambda^{-2}$.  Strictly speaking the step distribution
depends on the energy shell $E$, so that $p_t(x) = |e^{itH}_{0x}|^2$ should resemble not a single Gaussian but a mixture of Gaussian distributions with different variances, all on the order $\lambda^{-2}t$.  Nevertheless, it is not a bad heuristic to imagine that
\begin{equation}
\label{eq:gaussian-approximation}
|e^{itH}_{0x}|^2 \approx (\lambda^{-2}t)^{-d/2} \exp(-\lambda^2 |x|^2/t).
\end{equation}

For reasons that we will see later, it is more convenient to work with the resolvent $R(z) = (H-z)^{-1}$ than with the propagator $e^{-itH}$.  We write $z=E+i\eta$, where $E$ selects the energy and
$\eta\to 0$ is a small parameter.  As a function of $E$, $f(x) = \frac{1}{x - E + i\eta}$ is smooth to scale $\eta$, so by Corollary~\ref{cor:projection-kinetic} we have $R(E+i\eta)\approx R_0(E+i\eta)$
for $\eta \gg \lambda^2$.  Indeed, $R(z)$ has a nice explicit expression in terms of the propagator:
\begin{equation}
\label{eq:prop-to-R}
R(z) = \int_0^\infty e^{iEt} e^{-\eta t} e^{itH} \diff t.
\end{equation}
A heuristic for the entries of $|R(z)_{0x}|^2$ is that the entries of $e^{itH}_{0x}$ are completely decorrelated.  Combined with~\eqref{eq:gaussian-approximation} we arrive at the guess
\begin{align*}
|R(z)_{0x}|^2
&\approx \int_0^\infty e^{-2\eta t} |e^{itH}_{0x}|^2\diff t  \\
&\approx \int_0^\infty e^{-2\eta t} (\lambda^{-2}t)^{-d/2} \exp(-\lambda^2|x|^2/t)\diff t.
\end{align*}
Taking the limit $\eta\to 0$ we obtain, for $d\geq 3$, the heuristic
\begin{equation}
\label{eq:greens-approx}
|R(z)_{0x}|^2 \approx \lambda^2|x|^{2-d}.
\end{equation}
Although~\eqref{eq:greens-approx} and~\eqref{eq:gaussian-approximation} are heuristically related, it is not clear that either implies the other.  Nevertheless, we do have by applying
Parseval's theorem to the identity~\eqref{eq:prop-to-R}
\[
\eta \int_0^\infty e^{-\eta t} |e^{itH}_{0x}|^2 \diff t
= \int_{-\infty}^\infty |R_{0x}(E+i\eta)|^2 \diff E.
\]
Thus by summing over an annulus at the diffusive scale $\lambda^{-1}\eta^{-1/2}$ we can deduce \textit{time-averaged} information about $r(t)$ from the distribution of the mass in the rows of
the resolvent.

In particular, Theorem~\ref{thm:main-thm} is a consequence of the following estimate.
\begin{theorem}
\label{thm:resolvent-deloc}
For any $\eps,\delta>0$ there exists $c,C$ such that for $E\in[-2d,2d]$ with $d(E,\Sigma_d)>\eps$
and $\eta > \lambda^{2.1-\delta}$,
\[
\sum_{|x|>c\lambda\eta^{-1/2}} |R_{0x}(E+i\eta)|^2 \geq c \eta^{-1}
\]
holds with probability at least $1-C\lambda^{100}$.
\end{theorem}

\subsection{The direct perturbative approach}
As a first try, let's see how one might go about understanding the resolvent directly from perturbation
theory.  To this end we use the resolvent identity,
\[
(H_0+W-z)^{-1} = (H_0-z)^{-1} - (H_0-z)^{-1} W (H_0+W-z)^{-1}.
\]
This can be checked by multiplying both sides on the right by $(H_0+W-z)$.   Applying this identity
with $H_0 = \Delta$ and $W=\lambda V$ we arrive at the identity
\[
(H - z)^{-1} = (\Delta-z)^{-1} - (\Delta-z)^{-1}(\lambda V) (H-z)^{-1}.
\]
Writing $R(z) = (H-z)^{-1}$ and $R_0(z)=(\Delta-z)^{-1}$ and iterating this formula, we arrive
at the \textit{Born series} expansion for $R(z)$,
\[
R(z) = R_0(z) + \sum_{j=1}^\infty (-1)^j R_0(z) (\lambda V R_0(z))^j.
\]
This is algebraically much simpler than the Dyson series expansion we used in the previous
chapter, and this is the main reason to use $R(z)$ instead of the propagator $e^{itH}$.

Note that the $j$-th term in the expansion above can be rearranged to
\[
R_0^{1/2} (\lambda R_0^{1/2} V R_0^{1/2})^j R_0^{1/2}.
\]
The term in brackets is a random matrix that is linear in the randomness, and is in a form that
is amenable to applying the non-commutative Khintchine inequality.  The naive bound for the operator,
using $\|R_0^{1/2}\|_{\ell^2\to\ell^2} = \eta^{-1/2}$ (remember that $z=E+i\eta$, $\eta>0$), is
$\|R_0^{1/2}VR_0^{1/2}\| \lsim \eta^{-1}$.  However, by applying the non-commutative Khintchine inequality, one can obtain a square-root cancellation for this operator, and instead get
\[
\|R_0^{1/2}VR_0^{1/2}\|_{\ell^2\to\ell^2}\lsim \eta^{-1/2}.
\]
Therefore, for $\eta \gg \lambda^2$, the term $\lambda R_0^{1/2} V R_0^{1/2}$ is of lower order than
the main term $R_0(z)$.

Taking $\eta$ into the diffusive regime $\eta \ll\lambda^{2}$ requires \textit{renormalization}.
In this context, renormalization is simply the fact that one is not forced to take $H_0=\Delta$,
but could instead take $H_0=\Delta + \Theta$ for any operator $\Theta$ we choose.  It turns out
that for the Anderson model, we can take $\Theta = \lambda^2 \theta\Id$ for a scalar $\theta$ (we will
see how one might renormalize more general models soon).  The scaling $\lambda^2$ is for convenience
(it will turn out that $\theta$ is an $O(1)$-sized quantity).

To explain the following calculation it is useful to introduce a diagrammatic notation.  In this notation, the
resolvent is represented by a solid line,
\[
R = \ER[0.6],
\]
and we define $M=(\Delta - (z+\lambda^2\theta))^{-1}$, represented by a dashed line,
$M=$\begin{tikzpicture} \draw [semithick, densely dashed] (0,0.1) -- (1,0.1); \draw[white] (0,0)--(1,0); \end{tikzpicture}.
With this choice of renormalization we have the Born series expansion
\begin{equation}
\label{eq:R-expansion}
R(z) = M + \sum_{j=1}^\infty (-1)^j M
((\lambda V - \lambda^2\theta)M)^j.
\end{equation}
Diagrammatically, this is represented by
\[
\ER = \snake + \snakeO + \snakeT + \cdots,
\]
where we use the circled cross to represent $\lambda V - \lambda^2\theta$ (the cross represents $\lambda V$ and the circle $\lambda^2\theta$).
Upon taking an expectation the term with a single cross vanishes and the two collisions in the third term
are paired by the Wick rule, so the first few terms can be written as
\[
\Expec
\ER = \snake +
\begin{tikzpicture}
\draw [semithick, densely dashed] (0,0) -- (2,0);
\node [circle, draw=black, fill=white, inner sep=2pt, line width=0.7pt] at (1,0) {};
\end{tikzpicture}
+
\begin{tikzpicture}
\draw [semithick, densely dashed] (0,0) -- (2,0);
\node [cross,fill=black,inner sep=1.5pt, line width=1pt] (X) at (0.6,0) {};
\node [cross,fill=black,inner sep=1.5pt, line width=1pt] (Y) at (1.4,0) {};
\arcto{(X)}{(Y)}{semithick}
\end{tikzpicture}
+ \cdots.
\]

The latter two terms can be made to cancel if we choose $\theta$ to solve the self-consistent equation
\begin{equation}
\label{eq:SCE-diagram}
\SCE[0.6]  +
\begin{tikzpicture}[scale=0.6]
\coordinate (Q) at (0,0);
\node [circle, draw=black, fill=white, inner sep=2pt, line width=0.7pt] at (Q) {};
\end{tikzpicture} = 0,
\end{equation}
which can also be written as
\[
\lambda^2 \Expec VRV =
\lambda^2 \theta \,\Id,
\]
and taking a diagonal entry this produces the following defining equation for $\theta$:
\[
\theta = M_{00} = (\Delta - (z+\lambda^2\theta))^{-1}_{00}.
\]
Returning to the diagrammatic interpretation of the self-consistent equation, we may use it to
simplify away self-loops in more complicated terms coming from the Wick expansion of higher order terms
in the Born series.  For example, using~\eqref{eq:SCE-diagram} we have the identity
\[
\begin{tikzpicture}[scale=0.6]
\coordinate (A) at (0,0);
\coordinate (Q) at (1,0);
\coordinate (R) at (2,0);
\coordinate (S) at (3,0);
\coordinate (T) at (4,0);
\coordinate (U) at (5,0);
\coordinate (V) at (6,0);
\coordinate (B) at (7,0);
\draw [semithick, densely dashed] (A) -- (B);
\node [cross, fill=black, inner sep=1.5pt, line width=1pt] at (Q) {};
\node [cross, fill=black, inner sep=1.5pt, line width=1pt] at (R) {};
\node [cross, fill=black, inner sep=1.5pt, line width=1pt] at (S) {};
\node [cross, fill=black, inner sep=1.5pt, line width=1pt] at (T) {};
\node [cross, fill=black, inner sep=1.5pt, line width=1pt] at (U) {};
\node [cross, fill=black, inner sep=1.5pt, line width=1pt] at (V) {};
\arcto{(Q)}{(T)}{semithick}
\arcto{(R)}{(S)}{semithick}
\arcto{(U)}{(V)}{semithick}
\end{tikzpicture}
+
\begin{tikzpicture}[scale=0.6]
\coordinate (A) at (0,0);
\coordinate (Q) at (1,0);
\coordinate (X) at (2.5,0);
\coordinate (T) at (4,0);
\coordinate (U) at (5,0);
\coordinate (V) at (6,0);
\coordinate (B) at (7,0);
\draw [semithick, densely dashed] (A) -- (B);
\node [cross, fill=black, inner sep=1.5pt, line width=1pt] at (Q) {};
\node [circle, draw=black, fill=white, inner sep=2pt, line width=0.7pt] at (X) {};
\node [cross, fill=black, inner sep=1.5pt, line width=1pt] at (T) {};
\node [cross, fill=black, inner sep=1.5pt, line width=1pt] at (U) {};
\node [cross, fill=black, inner sep=1.5pt, line width=1pt] at (V) {};
\arcto{(Q)}{(T)}{semithick}
\arcto{(U)}{(V)}{semithick}
\end{tikzpicture}
= 0.
\]
The terms that do \textit{not} cancel involve crossings.  There are \textit{many} crossing terms,
 the simplest of which is given by
\[
\begin{tikzpicture}[scale=0.6]
\coordinate (A) at (0,0);
\coordinate (Q) at (1,0);
\coordinate (R) at (2,0);
\coordinate (S) at (3,0);
\coordinate (T) at (4,0);
\coordinate (B) at (5,0);
\draw [semithick, densely dashed] (A) -- (B);
\node [cross, fill=black, inner sep=1.5pt, line width=1pt] at (Q) {};
\node [cross, fill=black, inner sep=1.5pt, line width=1pt] at (R) {};
\node [cross, fill=black, inner sep=1.5pt, line width=1pt] at (S) {};
\node [cross, fill=black, inner sep=1.5pt, line width=1pt] at (T) {}; 
\arcto{(Q)}{(S)}{semithick}
\arcto{(R)}{(T)}{semithick}
\end{tikzpicture}
= \lambda^4\sum_{x,y} M \ket{x}M_{xy}^2 M_{yx} \bra{y}M
\]
A ``naive'' estimate would indicate that this term has size $\lambda^4 \eta^{-2}$, but in fact the
crossing induces an additional cancellation such that the term actually has size $\lambda^4\eta^{-1}$
(although this is delicate).  In~\cite{ESYAnderson,ESYrecollision} used
combinatorial arguments and estimates from harmonic analysis to show that the sum of all crossing terms has size
on the order $\lambda^{2+\eps} \eta^{-1}$.  In fact, they handled second moments of the terms in
the \textit{Dyson} series which is even more complicated, but their methods should of course apply
to the simpler problem of computing the expectation of the resolvent.

\subsection{The resolvent of a Wigner matrix}
We are not going to try to directly estimate the crossing terms.
Instead we take inspiration from random matrix theory, where crossing terms also appear but elegant methods have been developed to handle them.
In particular, we use the framework of \textit{self-consistent equations} which was developed to
understand the local spacing of eigenvalues in Wigner matrices~\cite{erdos2009semicircle}.  An excellent exposition of this approach is provided in~\cite{erdos2013local}.  We will present a suboptimal result which nevertheless suffices
for our purposes.

The simplest example of a random matrix is the GOE ensemble, having independent (symmetric) Gaussian entries of variance $1$ above the diagonal and variance $2$ on the diagonal.  We can express such matrices as
\begin{equation}
\label{eq:HGOE}
H_{\rm GOE} = \frac{1}{\sqrt{N}}\sum_{1\leq i\leq j\leq N} g_{ij} E_{ij},
\end{equation}
where
\[
E_{ij} =
\begin{cases}
\sqrt{2} e_i\otimes e_i, & i=j,\\
e_i\otimes e_j + e_j\otimes e_i, & i\not= j.
\end{cases}
\]
For convenience we use the bra-ket notation $\ket{i}\bra{j}$ to mean $e_i\otimes e_j$.

The distribution of $H_{\rm GOE}$ is invariant under conjugation by a rotation matrix, so the eigenvectors are
uniformly distributed from the unit sphere.  What is more interesting is to compute the distribution of
eigenvalues of $H_{\rm GOE}$.  Let $\lambda_k$ be the set of eigenvalues and $\psi_k$ the eigenvectors, so that  $H_{\rm GOE}\psi_k = \lambda_k \psi_k$.  We define the empirical measure of eigenvalues $\mu$ by
\[
\mu = \sum_k \delta_{\lambda_k}.
\]
Then we can write
\[
R_{\rm GOE}(z) = (H_{\rm GOE}-z)^{-1} = \sum_k \frac{1}{\lambda_k-z} \psi_k \otimes\psi_k.
\]
Taking a trace, we have
\[
\trace[ R_{\rm GOE}(z) ] = \sum_k \frac{1}{\lambda_k-z} = \int \frac{1}{t-z} \diff \mu(t).
\]
Writing $z=E+i\eta$ with $\eta>0$ and taking the imaginary part, we have
\[
-\Impt \trace[R_{\rm GOE}(z)] = \int \frac{\eta}{(E-t)^2 + \eta^2} \diff \mu(t).
\]
The right hand side is a convolution of $\mu$ with the Poisson kernel, which
(up to a scaling by $\pi$), is an approximate identity.  More precisely, $R_{\rm GOE}(E+i\eta)$ encodes (a
smoothed version of) the eigenvalue count in the interval $[E-\eta,E+\eta]$.  It is therefore of interest
to compute the diagonal elements of $R_{\rm GOE}(E+i\eta)$ with $\eta$ as small as possible.

An efficient way to compute $R_{\rm GOE}$ is to use the ``cavity method''.  That is, one
introduces a ``cavity'' by deleting one site from $[N]$.  This corresponds to deleting a row and a column
of the matrix, resulting in a sample from the $(N-1)\times (N-1)$ GOE ensemble.  The Schur complement formula can be used to relate the resolvent from the smaller ensemble to the resolvent of the full ensemble.  Moreover, these resolvents can be shown to be closely related (for example, one has the eigenvalue interlacing property).  This strategy uses strongly the symmetries of the model and is not well suited for working with the Anderson model (where deleting one site of $\bbZ^d$ results in a different graph that no longer has translation symmetry).

Fortunately there is an alternative approach that uses only (1) the resolvent identity and (2) Gaussian integration by parts.   First we use the resolvent identity using the renormalization $H-z = -(m(z)\Id+z) + (H+m(z)\Id)$ (we drop the $\Id$ and the dependence on $z$ below):
\begin{equation}
\label{eq:goe-resolvent-id}
R = -(m+z)^{-1} + (m+z)^{-1} (H+m) R.
\end{equation}
At this point one \textit{could} iterate this formula and obtain an expansion for $R_{\rm GOE}$,
resulting in an expression analogous to~\eqref{eq:R-expansion}.  However there is an alternative approach
using the Gaussian integration by parts formula $\Expec Z_i f(Z) = \Expec \partial_i f(Z)$.  Indeed,
we can use the definition of $H_{\rm GOE}$~\eqref{eq:HGOE} and the resolvent identity to compute
\[
\frac{\partial}{\partial g_{ij}} R = -\frac{1}{\sqrt{N}} R E_{ij} R.
\]
Thus, taking an expectation on both sides of~\eqref{eq:goe-resolvent-id} we obtain
\begin{align*}
\Expec R
&= -(m+z)^{-1} + (m+z)^{-1} (\Expec HR + \Expec mR) \\
&= -(m+z)^{-1} + (m+z)^{-1} (-N^{-1}\sum_{i\leq j} \Expec E_{ij}RE_{ij}R + \Expec mR).
\end{align*}
To make sense of this we define the superoperator $\mcal{A}_{\rm GOE}$
\begin{align*}
\mcal{A}_{\rm GOE}[B] &:= N^{-1} \sum_{i\leq j} E_{ij} B E_{ij} \\
&= N^{-1} \sum_{i} 2 B_{ii} e_i\otimes e_i
+ N^{-1} \sum_{i\not=j} B_{ii} e_j\otimes e_j + B_{ij} e_j\otimes e_i \\
&= N^{-1} B^\intercal + N^{-1}\trace[B] \Id.
\end{align*}

Therefore we have
\begin{equation}
\label{eq:goe-loc-appx}
\Expec R = -(m+z)^{-1}(\Id + N^{-1} \Expec R^\intercal R + \Expec (N^{-1}\trace[R]-m) R).
\end{equation}
Using that $\|R\|_{op}\leq \eta^{-1}$, we have that $N^{-1}\Expec R^\intercal R$ is bounded by
$N^{-1}\eta^{-2}$ in operator norm (which is small so long as we take $\eta \gg N^{-1/2}$).  We
can make the final term small if we take $m = \Expec N^{-1}\trace R = \Expec R_{00}$ and if we can
show that $N^{-1}\trace R$ is concentrated around its mean.    We will see that this is the case
in a moment, but first let us complete the calculation.  Taking the diagonal entries of
\eqref{eq:goe-loc-appx} above and setting $m = \Expec R_{00}$, we have that $m$ solves
\[
m = -(m+z)^{-1} (1 + \eps)
\]
where
\[
\eps := N^{-1}\Expec (R^\intercal R)_{00} + \Expec (N^{-1}\trace[R-\Expec R])R_{00}.
\]
Using the quadratic formula, the solution $m$ is given by
\[
m = \frac{-z \pm \sqrt{z^2 - 4(1+\eps)}}{2}.
\]
Now if we set $z=E+i\eta$ for small $\eta\ll 1$, we obtain (neglecting $\eta$ and $\eps$)
\[
\Impt m \approx \Impt \sqrt{E^2/4 - 1} = \begin{cases} \sqrt{1 - (E/2)^2}, & |E|\leq 2 \\ 0, & |E|>2.\end{cases}
\]
This is the semicircular law.

To complete the proof we need to show that $N^{-1}\trace[R]$ is concentrated around its mean.  To simplify the calculation we will actually show that $R_{00}$ is concentrated, and this is done
by computing the Lipschitz constant of $R_{00}$ as a function of $g_{ij}$.  We compute:
\begin{align*}
|\nabla R_{00}|^2
&= \sum_{i\leq j} \Big|\frac{\partial}{\partial g_{ij}} R_{00}\Big|^2 \\
&= N^{-1}\sum_{i\leq j} |(R E_{ij} R)_{00}|^2 \\
&\lsim N^{-1}\sum_{i,j} |R_{0i}|^2 |R_{0j}|^2 \\
&= N^{-1} \Big(\sum_i |R_{0,i}|^2\Big)^2 \\
&= N^{-1} \eta^{-2} (\Impt R_{00})^2.
\end{align*}
Deterministically we have $\Impt R_{00} \leq \eta^{-1}$, so we have
\[
|\nabla R_{00}|^2 \lsim \eta^{-4}N^{-1}.
\]
Therefore,
\[
\Prob( |R_{00} - \Expec R_{00}| \geq K \eta^{-2}N^{-1/2}) \leq \exp(-cK^2).
\]
We have therefore proved the following:
\begin{proposition}
For $H_{\rm GOE}$ as in~\eqref{eq:HGOE} and $z=E+i\eta$, with $E\in(-2,2)$ we have that
\[
\Prob(|\Impt \trace R(z) - \sqrt{1-(E/2)^2}| \geq K \eta^{-2} N^{-1/2}) \leq \exp(-cK^2).
\]
\end{proposition}
This kind of result is called a ``local law'', as it provides information about the local density of eigenvalues at
a microscopic scale.  The optimal local law provides information for imaginary parts $\eta >> N^{-1}$, and thus
nearly captures the location of individual eigenvalues.  The bound above only provides information down to scale $\eta \gg N^{-1/4}$.

\subsubsection{More general ensembles}
The nice thing about the derivation of the semicircular law that we just saw is that it is rather robust,
not relying on any particular symmetry of the GOE ensemble.  Let's go through the calculation again,
this time for a more general Gaussian random matrix of the form
\[
H = A_0 + \sum_j g_j A_j.
\]
In the case of Wigner random matrices we have $A_0=0$ and $A_{ij} = E_{ij}$, and in the case of
the Anderson tight-binding model we will take $A_0=\Delta_{\bbZ^d}$ and $A_j = \ket{j}\bra{j}$.  Other lattices can be accomodated for by changing $A_0$, and nontrivial covariance structure can be added to
the potential by changing $A_j$.

We renormalize by writing $H-z = (A_0 - (z+M)) + (\sum_j g_j A_j + M)$, and so by the resolvent
identity have
\[
R = (A_0 - (z+M))^{-1} - (A_0 - (z+M))^{-1} \Big(\sum_j g_j A_j + M\Big) R.
\]
Taking an expectation and performing Gaussian integration by parts we arrive at
\[
\Expec R = (A_0 - (z+M))^{-1} - (A_0 - (z+M))^{-1} (\Expec (M - \mcal{A}[R])R),
\]
where $\mcal{A}$ is the superoperator
\[
\mcal{A}[B] := \sum_j A_j B A_j.
\]
Setting $M = \mcal{A}[\Expec R]$  we arrive at the following approximate self-consistent equation for
$\Expec R$.
\begin{equation}
\label{eq:general-ER-eq}
\Expec R = (A_0 - (z+\mcal{A}[\Expec R]))^{-1} + (A_0 -(z+\mcal{A}[\Expec R]))^{-1}
(\Expec \mcal{A}[R-\Expec R]R).
\end{equation}
If $\mcal{A}[R]$ is sufficiently concentrated, we can treat the second term as an error term.  We then
arrive at an (approximate) self-consistent equation for $\Theta := \Expec \mcal{A}[R]$
\begin{equation}
\label{eq:general-theta-eq}
\Theta = \mcal{A}[(A_0 - (z+\Theta))^{-1}] + \mathfrak{E},
\end{equation}
where
\[
\mathfrak{E} = \mcal{A}[(A - (z+\Theta))^{-1} \Expec \mcal{A}[R-\Expec R]R].
\]

Before we specialize to the Anderson model, we stop to make an observation that applies generally.
The formal calculation we started with for the Anderson model suggests that $\mathfrak{E}$ should encode
the ``crossing'' terms, and it is not clear from the expression above what $\mathfrak{E}$ has to
to with crossings.  To see the connection, we will derive an alternative expression for
$\Expec \mcal{A}[R-\Expec R]R$.  Note first that, if $R'$ is an independent sample of the resolvent,
then
\[
\Expec \mcal{A}[R-\Expec R]R' = 0.
\]
To make use of this we interpolate between $R$ and $R'$ in the second factor.  Define
$R^q$ to be the resolvent of $H^q$, where
\[
H^q = A_0 + \sum_j (g_j \sqrt{1-q}+ g_j'\sqrt{q}) A_j,
\]
and $g_j'$ are independent samples of $g_j$.  Then $R = R^0$ and $R'=R^1$, so using the fundamental
theorem of calculus to interpolate as in~\cite{BBVH} we have
\begin{align*}
\Expec \mcal{A}[R-\Expec R]R
&= \Expec \mcal{A}[R^0-\Expec R^0]R^1
+ \int_0^1 \frac{d}{dq} \Expec \mcal{A}[R^q-\Expec R^q]R^1 \diff q \\
&= \int_0^1 \frac{d}{dq} \Expec \mcal{A}[R^q]R^1 \diff q.
\end{align*}
Then we compute, using the resolvent identity and another application of GIBP:
\begin{align*}
\frac{d}{dq} \Expec \mcal{A}[R^q]R^1 &=
\sum_j \frac{1}{2\sqrt{q}}\Expec g_j' \mcal{A}[R^qA_jR^q]R^1 -\frac{1}{2\sqrt{1-q}} \Expec g_j \mcal{A}[R^qA_jR^q]R^1 \\
&= \sum_j \Expec \mcal{A}[R^q A_j R^q] R^1 A_j R^1 \\
&= \sum_{j,k} \Expec A_k R^q A_j R^q A_k R^1 A_j R^1.
\end{align*}
The final expression above is a ``crossing'' term, because the pairing of the $A_j$ and $A_k$ are
crossed.  This explains that although $\mathfrak{E}$ encodes the contributions of all of the
crossing terms, it can be bounded without reference to any crossings at all.

\subsection{The local law for the Anderson model}
We are finally ready to start thinking about the Anderson model on $\bbZ^d$, defined by the
Hamiltonian
\[
H = \Delta + \lambda \sum_{j\in\bbZ^d} g_j \ket{j}\bra{j}.
\]
It is more convenient for us to work on the torus
$\bbZ^d_L := \bbZ^d / L\bbZ^d$, with $L = \lceil \lambda^{-100} \rceil$ (as in the previous chapter).

The superoperator $\mcal{A}$ corresponding to this matrix ensemble is $\lambda^2\mcal{D}$,
where the diagonal superoperator $\mcal{D}$ is defined by
\[
\mcal{D}[B] = \sum_i \ket{i}\braket{i|B|i}\bra{i} = \sum_i B_{ii} \ket{i}\bra{i}.
\]
Therefore, defining $\tilde{\theta}(z) := \Expec R_{00}(z)$, the general self-consistent equation~\eqref{eq:general-ER-eq} becomes
\[
\Expec R = (\Delta - (z+\lambda^2 \tilde{\theta}))^{-1}
+ \lambda^2 (\Delta - (z + \lambda^2\tilde{\theta}))^{-1} \Expec \mcal{D}[R-\Expec R]R.
\]
Taking the diagonal entry of both sides we arrive at the following approximate self-consistent
equation for $\tilde{\theta}$ (which is the analogue of~\eqref{eq:general-theta-eq}):
\begin{equation}
\label{eq:appx-theta}
\tilde{\theta} = (\Delta - (z+\lambda^2\tilde{\theta}))^{-1}_{00} + \lambda^2
((\Delta - (z+\lambda^2\tilde{\theta}))^{-1} \Expec \mcal{D}[R-\Expec R]R)_{00}.
\end{equation}
As a first step towards proving quantum diffusion we will show that, for $z=E+i\eta$ and $\eta \gg \lambda^{2+\frac16}$, $\tilde{\theta}$ is well approximated by the solution to
\begin{equation}
\label{eq:theta-sce}
\theta = (\Delta - (z+\lambda^2\theta))^{-1}_{00}.
\end{equation}

To state our result precisely it is convenient to introduce stochastic domination notation.  We say that $B$
stochastically dominates $A$, written $A\ldom B$, if for any $\eps,N>0$ there exists $C(\eps,N)$ such that
\[
A\leq C(\eps,N) \lambda^\eps B
\]
holds with probability at least $1-C(\eps,N)\lambda^N$.
\begin{theorem}
\label{thm:local-law}
Fix $\eps,\delta>0$, and
let $E\in[-2d,2d]$ with $d(E,\Sigma_d)>\eps$.
Then for $z=E+i\eta$ with $\eta \gg \lambda^{2+\frac16-\delta}$,
\[
|\Expec R_{00} - \theta| \leq \lambda^{1/2-\delta} (\lambda \eta^{-1})^3
\]
and moreover
\begin{equation}
\label{eq:Rxy-conc}
|R_{xy} - \Expec R_{xy}| \ldom \lambda^{1/2} (\lambda \eta^{-1})^2.
\end{equation}
\end{theorem}
Note that the above theorem does not say anything interesting about the values of the resolvent far away from
the diagonal, so it is not yet clear how this is useful for Theorem~\ref{thm:resolvent-deloc}.

There are two steps to establishing Theorem~\ref{thm:local-law}.
The first step is to understand the exact self-consistent equation
\[
\theta = (\Delta - (z+\lambda^2\theta))^{-1}_{00}.
\]
In particular, one would like to know that a solution $\theta$ exists, is unique, and is bounded before one can try to understand
what happens to a perturbation.

The second step is to bound the error term in~\eqref{eq:appx-theta}.  The important term is the matrix
$\mcal{D}[R-\Expec R]$, which is a diagonal matrix with entries $R_{xx}-\Expec R_{xx}$.  Thus we can obtain a bound in operator
norm for this matrix by proving a concentration inequality for the entries $R_{xx}$.
This concentration inequality is where we use the $\|R\|_{p\to q}$ bounds proven in the previous chapter.

\subsubsection{The self-consistent equation for $\theta$}
First we analyze the equation
\begin{equation}
\label{eq:theta-sce2}
\theta = (\Delta - (z+\lambda^2\theta))^{-1}_{00}.
\end{equation}
To this end it is useful to understand the function
\[
F(z) := (\Delta - z)^{-1}_{00}.
\]
Using the Fourier transform, one can confirm that  $F(z)$ is given by
\[
F(z) = \int_{\Torus^d} \frac{\diff \xi}{\omega(\xi) - z}.
\]
This is an analytic function of $z$, and using the coarea formula we can rewrite the imaginary part as follows:
\begin{align*}
\Impt F(z)
&= \int_{\Torus^d} \frac{\eta \diff \xi}{ (\omega(\xi)-E)^2 + \eta^2} \\
&= \int_{-2d}^{2d} \frac{\eta}{(E'-E)^2+\eta^2}
\Big(\int_{\{\omega(\xi)=E'\}} \frac{1}{|\nabla \omega(\xi)|} \diff \mcal{H}^{d-1}\Big) \diff E' \\
&=: \int_{-\infty}^{\infty} \frac{\eta}{(E'-E)^2 + \eta^2} \rho(E')\diff E',
\end{align*}
where we define the density of states $\rho$
\begin{equation}
\label{eq:rho-int}
\rho(E) := \int_{\{\omega(\xi)=E\}} \frac{1}{|\nabla \omega(\xi)|}\diff \mcal{H}^{d-1}.
\end{equation}
Thus,
\[
\rho(E) = \lim_{\eta\to 0} \frac1\pi\Impt F(E+i\eta).
\]
Since $F(z)$ is analytic, we have that
\begin{equation}
\label{eq:F-from-rho}
\lim_{\eta\to 0} \frac1\pi F(E+i\eta) = \Hilb \rho(E) + i\rho(E),
\end{equation}
where $\Hilb$ is the Hilbert transform.  We note the following properties of the function $\rho$ which can be verified
from the integral formula~\eqref{eq:rho-int}:
\begin{lemma}
The function $\rho$ given by~\eqref{eq:rho-int} is supported in $[-2d,2d]$.  Moreover, $\rho\in C^\infty([-2d,2d]\setminus\Sigma_d)$,
where the $\Sigma_{\rm crit}$ is the set of critical values  $\Sigma_{\rm crit} := [-2d,2d]\cap (2d + 4\bbZ)$.  Moreover,
$\rho>0$ on $(-2d,2d)$.  Finally, $\rho$ satisfies the bound
\[
\rho(E) \leq C_d(1 + \log(E^{-1})\One_{d=2}).
\]
\end{lemma}
Then using the formula~\eqref{eq:F-from-rho} and also the fact that $F(E+i\eta)$ is given by a convolution of $F(E+i0^+)$ with
the Poisson kernel, we have the following bounds for $F$:
\begin{lemma}
\label{lem:Freg}
For $\eta>0$, the function $F$ satisfies
\begin{equation}
\label{eq:Fbd}
|F(E+i\eta)| \lsim |\log\eta^{-1}|.
\end{equation}
Moreover,  on the domain $D_\eps := \{E+i\eta \mid \eta >0, E\in[-2d,2d], d(E,\Sigma_d) > \eps\}$, we have
\begin{equation}
\label{eq:Feq}
\|F\|_{C^1(D_\eps)} \leq C_\eps.
\end{equation}
\end{lemma}

As a consequence, we can prove a uniqueness result for solutions to~\eqref{eq:appx-theta}.
\begin{lemma}
\label{lem:theta-stability}
Let $z=E+i\eta$ for $E\in[-2d,2d]$, $d(E,\Sigma_d)>\eps$, and $\eta > \lambda^{10}$.
Suppose that $\theta' = \theta'(z)$ solves
\[
\theta' = F(z+\lambda^2 \theta') + \gamma
\]
for some error $|\gamma| < 1$.  Then for $\lambda < \lambda(\eps)$, we have the estimate
Then
\[
|\theta' - F(z)| \lsim |\gamma|.
\]
\end{lemma}
\begin{proof}
First, by~\eqref{eq:Fbd} we have the estimate
\[
|\theta'| \lsim \log\lambda^{-1}.
\]
But then, for $\lambda$ small enough it follows that $B_{\lambda^2\theta'}(z)\subset D_\eps$ where $D_\eps$ is the
domain defined in Lemma~\ref{lem:Freg}.  By the $C^1$ regularity estimate~\eqref{eq:Feq} it therefore follows that
\begin{align*}
|\theta' - F(z)|
&\leq |F(z+\lambda^2\theta')- F(z)| + |\gamma| \\
&\lsim  \lambda^2 |\theta'| + |\gamma|.
\end{align*}
Since $F(z)$ is bounded, we conclude that
\[
|\theta' - F(z)| \lsim \lambda^2 + |\gamma|.
\]
\end{proof}

\subsubsection{Concentration of the resolvent entries}
In this section we prove a concentration inequality for $R_{xy}$.  First we recall the Gaussian Poincar\'e inequality.

\begin{lemma}[Gaussian Poincar\'e]
For $f\in C^\infty(\Real^N)$ and $Z = (Z_1,\cdots,Z_N)$ $N$ independent standard Gaussian random variables,
\[
\operatorname*{Var} f(Z) \leq \Expec \sum_j |\partial_{Z_j} f(Z)|^2 =: \Expec |\nabla f|^2.
\]
\end{lemma}

Applying the Gaussian Poincare inequality to $|f|^k$ we have the following result for higher order moments:
\begin{lemma}
\label{lem:higher-GPI}
For $f\in C^\infty(\Real^N)$ with $\Expec f = 0$, we have for $k\geq 1$ the bound
\[
\Expec |f|^{2k} \leq (Ck)^{2k} \Expec |\nabla f|^{2k}.
\]
\end{lemma}
%
We will apply Lemma~\ref{lem:higher-GPI} to the resolvent entries $R_{xy}$, thought of as functions of the Gaussian potential
$g_x$.  Therefore we compute
\begin{equation}
\label{eq:R-conc}
\begin{split}
|\nabla R_{xy}|^2
&= \sum_w |\frac{\partial}{\partial g_w} R_{xy}|^2 \\
&= \lambda^2 \sum_w |R_{xw} R_{wy}|^2 \\
&\leq \lambda^2 (\sum_w |R_{xw}|^4)^{1/2} (\sum_w |R_{yw}|^4)^{1/2} \\
&\leq \lambda^2 \|R\|_{1\to 4}^4.
\end{split}
\end{equation}

By Theorem~\ref{thm:resolvent-apriori}, we have $\|R\|_{1\to 6} \ldom \lambda^2\eta^{-1}$ (we are taking $\eta \ll \lambda^2$).
Therefore (recalling $R=R(z)=R(E+i\eta)$),
\begin{align*}
\|R\|_{1\to 4}
&\leq \|R\|_{1\to 2}^{1/4} \|R\|_{1\to 6}^{3/4} \\
&\leq \|R^{1/2}\|_{1\to 2}^{1/4} \|R^{1/2}\|_{2\to 2}^{1/4} \|R\|_{1\to 6}^{3/4} \\
&\ldom (\lambda \eta^{-1/2})^{1/4} (\eta^{-1/2})^{1/4} (\lambda^2\eta^{-1})^{3/4} \\
\end{align*}

Combined with Lemma~\ref{lem:higher-GPI} and rearranging we conclude the following:

\begin{lemma}
\label{lem:Rconc}
Let $R = R(z)=R(E+i\eta)$ for a good energy $E$ and $\eta \ll \lambda^2$.  Then
\[
|R_{xy}-\Expec R_{xy}| \ldom \lambda^{1/2}(\lambda^2\eta^{-1})^2
\]
\end{lemma}

\subsubsection{Conclusion of the proof of Theorem~\ref{thm:local-law}}
We can now apply Lemma~\ref{lem:Rconc} to analyze the self-consistent equation~\eqref{eq:appx-theta} which we recall
\begin{equation*}
\tilde{\theta} = (\Delta - (z+\lambda^2\tilde{\theta}))^{-1}_{00} + \lambda^2
((\Delta - (z+\lambda^2\tilde{\theta}))^{-1} \Expec \mcal{D}[R-\Expec R]R)_{00}.
\end{equation*}
The error term on the right, $\mathfrak{e}$, is bounded as follows:
\begin{align*}
|\mathfrak{e}|
&= \lambda^2 |\braket{0|(\Delta-(z+\lambda^2\tilde{\theta}))^{-1} \Expec \mcal{D}[R-\Expec R] R| 0}| \\
&\leq \lambda^2 \|(\Delta - (z+\lambda^2\tilde{\theta}))^{-1}\|_{1\to 2} \Expec \|\mcal{D}[R-\Expec R] R\|_{1\to 2}.
\end{align*}
For the calculation below we will assume that we already know that $\tilde{\theta}=O(1)$ and that $\Impt\tilde{\theta} \gtrsim 1$.
These bounds are \textit{outputs} of our argument and at first this may appear circular,
and this circularity will be addressed shortly.
For now we will pretend we have this information already in hand.  To bound $\|(H-z)^{-1}\|_{1\to 2}$ we use the Ward
identity, valid for any symmetric $H$ and $\eta>0$:
\[
\sum_y |(H-z)^{-1}_{xy}|^2 = \eta^{-1} \Impt (H-z)^{-1}_{xx}.
\]
In particular,
\[
\|R\|_{1\to 2}^2 = \max_x \sum_y |R_{xy}|^2 \leq \eta^{-1} \max_x |R_{xx}|.
\]
Applying this with $H=\Delta$, and assuming that $\Impt\tilde{\theta}\gtrsim 1$ we therefore have
\[
\|(\Delta-(z+\lambda^2\tilde{\theta}))^{-1}\|_{1\to 2} \lsim \lambda^{-1}.
\]
Therefore
\begin{align*}
|\mathfrak{e}| &\lsim \lambda
\Big(\Expec \|\mcal{D}[R-\Expec R]\|_{2\to 2}^2\Big)^{1/2}
\Big(\Expec \|R\|_{1\to 2}^2\Big)^{1/2}  \\
&\lsim \lambda^{-\delta} \lambda \Big(\Expec (\sup_x |R_{xx} - \Expec R_{xx}|^2)\Big)^{1/2}
(\lambda \eta^{-1}) \\
&\lsim \lambda^{\frac12-\delta} (\lambda^2\eta^{-1})^3.
\end{align*}

In conclusion, we have
\[
\tilde{\theta} = (\Delta + (z+\lambda^2\tilde{\theta}))^{-1}_{00} + O(\lambda^{\frac12-\delta} (\lambda^2\eta^{-1})^3).
\]
Theorem~\ref{thm:local-law} now follows from Lemma~\ref{lem:theta-stability}.

\subsubsection{On the assumptions on $\tilde{\theta}$.}

What we have actually proved by the above argument is the following lemma.
\begin{lemma}
\label{lem:bootstrap}
Suppose that $\tilde{\theta}(z) := \Expec R(z)_{00}$ satisfies
$|\tilde{\theta}|\leq C$ and $\Impt\tilde{\theta}\geq c$.  Then for $\theta=\theta(z)$ solving~\eqref{eq:theta-sce2},
\[
|\tilde{\theta} - \theta| \lsim \lambda^{\frac12-\delta} (\lambda^2\eta^{-1})^3,
\]
\end{lemma}

For $z = E+i\eta$ with $\eta \gg \lambda^2$, the estimate $|\tilde{\theta}-\theta|\leq \eps$ holds as a consequence of
Corollary~\eqref{cor:projection-kinetic}.  This estimate starts the ``bootstrap'' -- down to $\eta > \lambda^{2+\frac16-\delta}$,
the bounds in the output of Lemma~\ref{lem:bootstrap} are stronger than the input.  At this point a completely correct proof of
Theorem~\ref{thm:local-law} follows from a continuity argument, using the fact that $\theta(z)$ is a continuous function of $z$.

\subsection{The $T$ equation}

Theorem~\ref{thm:local-law} specifies the entries of the resolvent $R_{xy}$ up to $\lambda^c$ absolute error.  However,
according to the heuristic~\eqref{eq:greens-approx} we should have $|R_{xy}|^2\sim \lambda^2|x-y|^{2-d}$ for entries
$|x-y|\gg \lambda^{-2}$.  As a consequence, most of the $\ell^2$ mass of the rows of the resolvent is concentrated
far from the origin in entries for which the estimate~\eqref{eq:Rxy-conc} has large relative error.
In particular, Theorem~\ref{thm:local-law} is not immediately helpful to understand sums of the form
\[
O[f] := \sum_x f(x) |R_{0x}|^2
\]
when $f$ is spread out on the diffusive scale, which is what is required to prove Theorem~\ref{thm:resolvent-deloc}.

To address this we need to write down a self-consistent equation for the second moment of $R$.  Notice first that $F$ can be
written in the form
\[
O[f] = (R F R^*)_{00},
\]
where $F$ is the diagonal matrix with entries $f(x)$, that is, $F = \sum_x f(x)\ket{x}\bra{x}$.  We define
\[
\tilde{M} := (\Delta - (z+\lambda^2\tilde{\theta}))^{-1},
\]
and then use the resolvent identity
\[
R = \tilde{M} - \tilde{M} (\lambda V + \lambda^2\tilde{\theta}) R
\]
to write
\[
RFR^* = (\tilde{M} - \tilde{M} (\lambda V + \lambda^2\tilde{\theta}) R) F R^*.
\]
Now we take an expectation and perform Gaussian integration by parts:
\begin{equation}
\label{eq:RFReq}
\begin{split}
\Expec RFR^* &= \tilde{M} F (\Expec R^*)
+ \lambda^2 \tilde{M} \Expec \mcal{D}[R-\tilde{\theta}]RFR^*
+ \lambda^2 \tilde{M} \Expec \mcal{D}[RFR^*] R^* \\
&= \tilde{M}F \tilde{M}^* + \lambda^2 \tilde{M} F \tilde{M}^* + \mathfrak{E}_{\rm T},
\end{split}
\end{equation}
where $\mathfrak{E}_{\rm T}$ is the error:
\begin{equation}
\label{eq:Terrdef}
\begin{split}
\mathfrak{E}_{\rm T} = &\tilde{M} F (\Expec R^* - \tilde{M})
+ \lambda^2 \tilde{M} \Expec \mcal{D}[R-\Expec R] RFR^* \\
&\qquad + \lambda^2 \tilde{M} \Expec \mcal{D}[RFR^*] (R^* - \tilde{M}^*).
\end{split}
\end{equation}
\textit{Remark:}
The equation~\eqref{eq:RFReq} is closely related to what physicists call the Bethe-Salpeter equation in the theory of wave propagation in random media~\cite{vollhardt1980diagrammatic}.

Admittedly, the equation~\eqref{eq:RFReq} is a bit opaque.
To interpret it properly, it is helpful to consider only the diagonal entries.  That is, we define
\[
g(x) := \Expec (RFR^*)_{xx}.
\]
In the case that $F = \sum_x f(x)\ket{x}\bra{x}$, the diagonal entries of the first term are given by a convolution,
\[
(\tilde{M}F\tilde{M}^*)_{xx} = \sum_y f(y) |\tilde{M}_{xy}|^2 := (\tilde{K} \ast f)(x),
\]
with kernel
\[
\tilde{K}(x) := |\tilde{M}_{0x}|^2.
\]
Defining the function $\mathfrak{e}_T(x) := (\mathfrak{E}_T)_{xx}$, we therefore obtain the following equation for $g$ in terms of
$f$:
\[
g = \tilde{K}\ast f + \lambda^2 \tilde{K}\ast g + \mathfrak{e}_{\rm T}.
\]
Define $\mcal{K}$ to be the operator that is convolution by $\tilde{K}$, so $\mcal{K}f = \tilde{K}\ast f$.  Then
\[
g = (\Id - \lambda^2\mcal{K})^{-1} \mcal{K}f + (\Id-\lambda^2\mcal{K})^{-1} \mathfrak{e}_{\rm T}.
\]
To understand what is going on with the main term, note that we can (at least formally) expand $(\Id - \lambda^2\mcal{K})^{-1}$
as a Neumann series:
\[
(\Id - \lambda^2\mcal{K})^{-1} = \Id + \sum_{j=1}^\infty (\lambda^2\mcal{K})^j.
\]
In Section~\ref{sec:finishing-pf} we will analyze this series expansion and compare it to a Green's function from a random walk.

In the remainder of this section we prove the following result.
\begin{proposition}
\label{prp:RFR}
For a bounded diagonal operator $F = \sum_x f(x)\ket{x}\bra{x}$ and $\lambda^{2.1-\delta} < \eta < \lambda^{2+\delta}$,
we have
\begin{equation}
\label{eq:RFR-bd}
|(RFR^*)_{xx} - (\Id -\lambda^2\mcal{K})^{-1} f| \ldom \lambda^{\frac12} (\lambda^2\eta^{-1})^5 \eta^{-1} \|f\|_{\ell^\infty}.
\end{equation}
\end{proposition}
Note that, taking $F=\Id$, we have by the Ward identity
\[
(RR^*)_{00} = \sum_x |R_{0x}|^2 = \eta^{-1} \Impt R_{00},
\]
so the error term on the right hand side of~\eqref{eq:RFR-bd} is small compard to the scale of $(RFR^*)_{xx}$ for a
macroscopic-sized observable (meaning, for $f$ having support on a diffusive-sized ball $\lambda^{-1}\eta^{-1/2}$).

The second point we make is that to prove~\eqref{eq:RFR-bd} it suffices to prove, for every $\delta>0$, the estimate
\[
|\Expec (RFR^*)_{xx} - (\Id -\lambda^2\mcal{K})^{-1} f| \lsim \lambda^{\frac12-\delta}
(\lambda^2\eta^{-1})^5 \eta^{-1} \|f\|_{\ell^\infty}.
\]
The concentration inequality then follows from the Gaussian poincare inequality and applications of the a priori $\|R\|_{p\to q}$ estimates.

The error in the expectation has the form
\[
(\Id - \lambda^2\mcal{K})^{-1} \mathfrak{e}_{\rm T}.
\]
Proposition~\ref{prp:RFR} then follows from the following two estimates.

\begin{lemma}
\label{lem:inftytoinfty}
The operator $(\Id-\lambda^2\mcal{K})^{-1}$ is bounded from $\ell^\infty\to\ell^\infty$, and in particular
\[
\|(\Id - \lambda^2\mcal{K})^{-1}\|_{\infty\to\infty} \leq C\lambda^2 \eta^{-1}.
\]
\end{lemma}

\begin{lemma}
\label{lem:ETbd}
The error $\mathfrak{E}_{\rm T}$ defined in~\eqref{eq:Terrdef} satisfies the estimate
\begin{equation}
\label{eq:ET-bd}
\|\mathfrak{E}_{\rm T}\|_{1\to \infty} =
\max_{x,y\in\bbZ^d_L}|(\mathfrak{E}_{\rm T})_{xy}| \lsim
[\lambda^{\frac12-\delta} (\lambda^2\eta{-1})^5] \lambda^{-2}
\|F\|_{2\to 2}.
\end{equation}
\end{lemma}

First we prove Lemma~\ref{lem:inftytoinfty}.
\begin{proof}[Proof of Lemma~\ref{lem:inftytoinfty}]
We estimate $\sum_x \tilde{K}(x)$ using the Ward identity and the self-consistent equation~\eqref{eq:theta-sce2} satisfied by $\theta$:
\begin{align*}
\lambda^2\sum_x \tilde{K}(x)
&= \lambda^2\sum_x |(\Delta-(z+\lambda^2\tilde{\theta}))^{-1}_{0x}|^2\\
&= \frac{\lambda^2\Impt (\Delta - (z+\lambda^2\tilde{\theta}))^{-1}_{00}}{\lambda^2\Impt\tilde{\theta}+\eta} \\
&= \frac{\lambda^2 \Impt (\Delta - (z+\lambda^2\theta))^{-1}_{00} + O(\lambda^{\frac52-\delta}(\lambda^2\eta^{-1})^3)}
{\lambda^2\Impt\tilde{\theta}+\eta} \\
&= \frac{\lambda^2 \Impt \theta + O(\lambda^{\frac52-\delta}(\lambda^2\eta^{-1})^3)}
{\lambda^2\Impt\theta + \eta + O(\lambda^{\frac52-\delta}(\lambda^2\eta^{-1})^3)} \\
&= 1 - (\Impt\theta)^{-1}\lambda^{-2}\eta + O(\lambda^{\frac12-\delta}(\lambda^2\eta)^5 \lambda^{-2}\eta)
\end{align*}
For $\eta \gg \lambda^{2.1-\delta}$, the last term is small compared to $\lambda^{-2}\eta$.
Therefore,
\[
\|\lambda^2\mcal{K}\|_{\infty\to\infty} \leq 1 - c\lambda^{-2}\eta.
\]
The result now follows for example by summing the Neumann series for $(\Id-\lambda^2\mcal{K})^{-1}$.
\end{proof}

The proof of Lemma~\ref{lem:ETbd} requires only slightly more calculation.
First we need the following lemma.
\begin{lemma}
For $d(E,\Sigma_d)>\eps$, $\eta>\lambda^{2+\frac16-\delta}$
\begin{equation}
\label{eq:MRonetwo}
\|\tilde{M} - \Expec R\|_{1\to 2} \lsim \lambda^{-\frac12-\delta} (\lambda^2\eta^{-1})^3.
\end{equation}
\end{lemma}
\begin{proof}
Recall that using the resolvent identity for $R$ and taking expectations we have
\[
\Expec R = \tilde{M} + \lambda^2 \tilde{M}\Expec \mcal{D}[R - \Expec R] R.
\]
The result follows upon using Lemma~\ref{lem:Rconc} and the estimate $\|\tilde{M}\|_{1\to 2} \lsim \lambda^{-1}$,
\begin{align*}
\|\Expec R - \tilde{M}\|_{1\to 2}
&\leq \lambda^2 \|\tilde{M}\|_{1\to 2}
\Expec \|\mcal{D}[R-\Expec R]\|_{2\to 2} \|R\|_{2\to 2}\\
&\lsim \lambda^{-\frac12-\delta} (\lambda^2\eta^{-1})^3.
\end{align*}
\end{proof}

We can now bound $\mathfrak{E}_{\rm T}$.
\begin{proof}[Proof of Lemma~\ref{lem:ETbd}]
We write
\[
\mathfrak{E}_{\rm T} =: \rm{I} + \rm{II} + \rm{III} + \rm{IV},
\]
where
\begin{align*}
\rm{I} &= \tilde{M} F (\Expec R^* - \tilde{M}) \\
\rm{II} &= \lambda^2 \tilde{M} \Expec \mcal{D}[R-\Expec R] RFR^* \\
\rm{III} &= \lambda^2 \tilde{M} (\Expec \mcal{D}[RFR^*]) \Expec(R^* - \tilde{M}^*) \\
\rm{IV} &= \lambda^2 \tilde{M} \Expec (\mcal{D}[RFR^*]-\Expec \mcal{D}[RFR^*]) (R^* - \tilde{M}^*).
\end{align*}
The first term is bounded using~\eqref{eq:MRonetwo},
\begin{align*}
\|\rm{I}\|_{1\to\infty}
&\leq \|\tilde{M}\|_{1\to 2} \|F\|_{2\to 2} \|\Expec R^* -\tilde{M}\|_{1\to 2} \\
&\lsim \lambda^{-\frac32-\delta} \|F\|_{2\to 2}  (\lambda^2\eta^{-1})^3.
\end{align*}
For the second term use Theorem~\ref{thm:local-law},
\begin{align*}
\|\rm{II}\|_{1\to\infty} &\leq
\lambda^2 \|\tilde{M}\|_{1\to 2} \Expec \|\mcal{D}[R-\Expec R]\|_{2\to 2} \|R\|_{2\to 2} \|F\|_{2\to 2} \|R^*\|_{1\to 2} \\
&\lsim \lambda (\lambda^{\frac12-\delta})(\lambda^2\eta^{-1})^2 \eta^{-1} \|F\|_{2\to 2} (\lambda \eta^{-1}) \\
&\lsim \lambda^{-\frac32-\delta} (\lambda^2\eta^{-1})^4  \|F\|_{2\to 2}.
\end{align*}
For the third term we use $\|\mcal{D}[RFR^*]\|_{2\to 2} \leq \|R\|_{1\to2}^2 \|F\|_{2\to2}$
as well as~\eqref{eq:MRonetwo}:
\begin{align*}
\|\rm{III}\|_{1\to\infty} &\leq \lambda^2 \|\tilde{M}\|_{1\to 2}
\Expec \|\mcal{D}[RFR^*]\|_{2\to 2} \|\Expec R-\tilde{M}\|_{1\to 2} \\
&\leq \lambda^2 \|\tilde{M}\|_{1\to 2}
\Expec \|R\|_{1\to 2}^2  \|F\|_{2\to 2} \|\Expec R-\tilde{M}\|_{1\to 2} \\
&\lsim \lambda (\lambda\eta^{-1})^2 \|F\|_{2\to 2} \lambda^{-\frac12-\delta}(\lambda^2\eta^{-1})^3 \\
&= \lambda^{-\frac32-\delta} (\lambda^2\eta^{-1})^5 \|F\|_{2\to 2}.
\end{align*}
For the fourth term we need an estimate for $\|\mcal{D}[RFR^*]-\Expec\mcal{D}[RFR^*]\|$, which is the maximum size of the diagonal fluctuations,
\[
\|\mcal{D}[RFR^*]-\Expec\mcal{D}[RFR^*]\| =
\sup_x |(RFR^*)_{xx} - \Expec (RFR^*)_{xx}|.
\]
A similar calculation as in~\eqref{eq:R-conc} shows that\footnote{Using the bounds for $\|R^{1/2}\|_{1\to 2}$,
$\|R^{1/2}\|_{2\to 2}$, and $\|R^{1/2}\|_{2\to 6}$ and interpolating as before.  This calculation has been omitted but is largely similar to the calculation for $\nabla R_{xx}$}
\begin{align*}
|\nabla (RFR^*)_{xx}|^2 &\leq \lambda^2 \|F\|_{2\to 2}^2 \|R\|_{1\to 4}^2 \|R\|_{1\to 2}^2\|R\|_{2\to 4}^2 \\
&\ldom \lambda^{-3} (\lambda^2\eta^{-1})^6
\end{align*}
Therefore,
\[
\|\mcal{D}[RFR^*] - \Expec \mcal{D}[RFR^*]\|_{2\to 2} \ldom \lambda^{-3/2} \|F\|_{2\to 2} (\lambda^2\eta^{-1})^3.
\]
Combined with the estimate $\|R\|_{1\to 2} \ldom \lambda \eta^{-1}$ we have
\begin{align*}
\rm{IV}
&\lsim \lambda^{-\delta} \lambda^2 \|\tilde{M}\|_{1\to 2} (\lambda^{-3/2}(\lambda^2\eta^{-1})^3)
(\lambda\eta^{-1}) \|F\|_{2\to 2} \\
&\lsim \lambda^{-\frac32-\delta} (\lambda^2\eta^{-1})^4 \|F\|_{2\to 2}
\end{align*}
which is the same scale as the bound for $\rm{II}$.

Rearranging the terms yields the estimate~\eqref{eq:ET-bd}.
\end{proof}

\subsection{Finishing the proof of Theorem~\ref{thm:resolvent-deloc}}
\label{sec:finishing-pf}

In this section we complete the proof of Theorem~\ref{thm:resolvent-deloc}, which states that the following bound holds
with high probability:
\[
\sum_{|x|>c\lambda\eta^{-1/2}} |R_{0x}(E+i\eta)|^2 \geq c \eta^{-1}.
\]
By the Ward identity, we have that
\[
\sum_x |R_{0x}|^2 \gtrsim \eta^{-1},
\]
so it suffices to show that for any $c_0$ we can find $c_1$ such that
\[
\sum_{|x|<c_1\lambda\eta^{-1/2}} |R_{0x}(E+i\eta)|^2 \leq c_0 \eta^{-1}
\]
holds with high probability.
By Proposition~\ref{prp:RFR} it then suffices to show that there exists $c_1$ such that
\begin{equation}
\label{eq:RW-bd}
((\Id - \lambda^2\mcal{K})^{-1} \mcal{K} \One_{\{|x|<c_1\lambda\eta^{-1/2}\}})(0) < \frac12 c_0 \eta^{-1}.
\end{equation}

We need more information about $\mcal{K}$ to estimate $(\Id-\lambda^2\mcal{K})^{-1}$ well enough to compute the main term.  In the paper~\cite{black2025self}, we think of
$(\Id - \lambda^2\mcal{K})^{-1}$ as an elliptic operator.  Equivalently, by the Neumann
series expansion we can think of it as the Green's function of a random walk.  In either
case, we need to estimate moments of $\lambda^2\mcal{K}$.
\begin{lemma}
\label{lem:K-moments}
The convolution kernel $\tilde{K}$ satisfies the following estimates:
\begin{align}
\label{eq:Kmass}
&\lambda^2 \sum \tilde{K}(x) \leq 1 - c\lambda^{-2}\eta \\
\label{eq:Kone}
&\lambda^2 \sum x_j\tilde{K}(x) = 0\\
\label{eq:Ktwo}
&\lambda^2 \sum |x|^2 \tilde{K}(x) \gtrsim \lambda^{-4} \\
\label{eq:Kfour}
&\lambda^2 \sum |x|^4 \tilde{K}(x) \lsim \lambda^{-8}.
\end{align}
\end{lemma}
\begin{proof}
The first bound~\eqref{eq:Kmass} was proven in the calculation above Lemma~\ref{lem:inftytoinfty}.
The identity~\eqref{eq:Kone} follows from the fact that $\tilde{K}$ is symmetric about the origin.

Now we prove~\eqref{eq:Ktwo}.
\[
(\Delta-z)^{-1}_{0x} = \int_{\Torus^d} e^{i\xi\cdot x} \frac{1}{\omega(\xi)-z} \diff \xi.
\]
For simplicity we set $w = z+\lambda^2\tilde{\theta}$ below.  Then by Plancherel's formula and then the coarea formula we have
\begin{align*}
\lambda^2 \sum |x|^2 \tilde{K}(x)
&= \lambda^2 \sum |x (\Delta-w)^{-1}_{0x}|^2  \\
&= \lambda^2 \int_{\Torus^d} \Big|\nabla (\omega(\xi)-w)^{-1}\Big|^2 \diff \xi \\
&= \lambda^2 \int_{\Torus^d} \frac{|\nabla \omega|^2}{|\omega(\xi)-w|^4}\diff \xi \\
&= \lambda^2 \int \frac{1}{|E'-w|^4} \Big(\int_{\{\omega(\xi)=E'\}} |\nabla \omega|\diff\mcal{H}^{d-1}\Big)\diff E'
\end{align*}
The function in brackets is positive, and $\int |E'-w|^{-4}\diff E \simeq \lambda^{-6}$, so this completes the proof
of~\eqref{eq:Ktwo}.

The proof of~\eqref{eq:Kfour} proceeds along similar lines (one uses Plancherel's theorem to produce a second derivative, otherwise
the use of the coarea formula is also similar).
\end{proof}

The main theorem now follows from the following elementary anticoncentration lemma for random walks.
\begin{lemma}
\label{lem:anticonc}
Let $X\in\Real^d$ be a random variable satisfying $\Expec X = 0$,
$\Expec X_iX_j = \sigma^2 \delta_{ij}$, and $\Expec |X|^3 \leq C \sigma^{3}$.  Let
$Y_N := \sum_{j=1}^N X^{(j)}$  be a sum of $N$ independent copies of $X$.  Then $Y_N$ satisfies,
for any $y\in\Real^d$,
\[
\Prob(|Y_N - y| \leq \sigma) \leq CN^{-d/2}.
\]
In particular, for $r\geq \sigma$ it follows that
\[
\Prob(|Y_N - y| \leq r) \leq CN^{-d/2} (r/\sigma)^d
\]
\end{lemma}
\begin{proof}
We give a proof via characteristic functions. For any random variable $Z$
define its characteristic function $f_{Z}:\mathbb{R}^d\to \mathbb{C}$
by $f_{Z}(\xi) := \mathbb{E}e^{i\langle Z,\xi\rangle}$.

By scaling it suffices to prove the result for $\sigma = 1$. First we estimate $f_{X}$ near $0$.
By the moment assumptions on $X$ and a Taylor expansion we can estimate
\begin{align*}
|f_X(\xi)|
& \leq  |\mathbb{E}(1+i\langle X, \xi\rangle - \frac{1}{2}\langle X, \xi\rangle^2 + C|\langle X, \xi\rangle|^3)|\\
& \leq 1- \frac{1}{2}|\xi|^2 + C|\xi|^3\mathbb{E}|X|^3\\
& \leq 1- \frac{1}{2}|\xi|^2 + C|\xi|^3\\
\end{align*}
Hence for $c_0>0$ small,  we have
$$|f_{X}(\xi)|\leq 1- \frac{1}{3}|\xi|^2,$$
for $|\xi|\leq c_0$. Now let  $\psi:\mathbb{R}^d\to \mathbb{R}$
satisfy $1_{|x|\leq 1}\leq \psi$ and $|\hat{\phi}|\leq 1_{|\xi|\leq c_{0}}$,
so that by Plancherel, and the estimate above we get
\begin{align*}
\mathbb{P}(|Y_N - y|\leq 1)
& \leq \mathbb{E}\psi(X)\\
& \leq \int_{\mathbb{R}^d}e^{iy\cdot \xi}\hat{\psi}(\xi)f_{Y_N}(\xi)d\xi\\
& \leq C \int_{|\xi|\leq c_0}|f_{Y_N}(\xi)|d\xi\\
& \leq C\int_{|\xi|\leq c_0}|f_{X}(\xi)|^N\\
& \leq C\int_{|\xi|\leq c_0} (1- \frac{|\xi|^2}{3})^{N}d\xi\\
& \leq CN^{-d/2}
\end{align*}
\end{proof}

We are now ready to establish~\eqref{eq:RW-bd}, and thereby prove Theorem~\ref{thm:main-thm}.
In fact we will prove a somewhat stronger statement that provides bounds for $\ell^2$-type averages
of the resolvent on scales much smaller than the diffusive scaling.
\begin{proposition}
Let $\eta > \lambda^{2.1-\delta}$ and
\[
r > C (\lambda^{1/2} (\lambda^2\eta^{-1})^{5/2}) \lambda^{-1} \eta^{-1/2}.
\]
Then letting $\chi_r$ be the indicator function for the ball of radius $r$,
\[
(\Id - \lambda^2\mcal{K})^{-1} (\mcal{K}\chi_r) (x) \lessapprox \lambda^2 r^2.
\]
\end{proposition}
\begin{proof}
By Proposition~\ref{prp:RFR} it suffices to prove that for any $x$
\begin{equation}
\label{eq:reduction-1}
(\Id-\lambda^2\mcal{K})^{-1} \mcal{K}\chi_r(x) \leq \lambda^2 r^2,
\end{equation}
where $\mcal{K}$ is convolution against the kernel $\tilde{K}_z$

Writing out the Neumann series for $(\Id-\lambda^2\mcal{K})^{-1}$ we have
\begin{align*}
(\text{Id}-\lambda^2\mathcal{K})^{-1}(\mathcal{K}\chi_{r}(x_0))(x)
&=  \lambda^{-2}\sum_{j=1}^\infty (\lambda^2 \tilde{K})^{\ast j}\ast \chi_r(x) \\
&=  \lambda^{-2} \sum_{j=1}^\infty (1-\alpha)^j ((1-\alpha)^{-1}\tilde{K})^{\ast j}\ast \chi_r(x),
\end{align*}
where we set $\alpha = 1-\lambda^2 \sum_x \tilde{K}(x)$ so that $(1-\alpha)^{-1}\tilde{K}$ is a
probability distribution (note that by Lemma~\ref{lem:K-moments}, $\alpha \simeq \lambda^{-2}\eta$).  Letting $Y_k = \sum_{j=1}^k X_j$ be the partial sums of a random
walk with independent steps $X_j$ of distribution $(1-\alpha)^{-1}\lambda^2 \tilde{K}$, we can write
the above as
\begin{align*}
(\text{Id}-\lambda^2\mathcal{K})^{-1}(\mathcal{K}\chi_{r}(x_0))(x)
&= \lambda^{-2} \sum_{j=1}^\infty (1-\alpha)^j \Prob(|Y_j - x_0| \leq r) \\
&\lsim \lambda^{-2} \sum_{j=1}^\infty (1-\alpha)^j \min\{1, r^d (\lambda^{-2} j)^{-d/2}\},
\end{align*}
The estimate~\eqref{eq:reduction-1} follows from combining Lemma~\ref{lem:K-moments} and
Lemma~\ref{lem:anticonc} (with $\sigma = \lambda^{-2}$).
For $j\leq \lambda^4 r^2$ we simply bound the terms by $1$ and get a contribution of
$\lambda^2 r^2$.  The terms with $j\gg \lambda^4 r^2$ are bounded similarly because of the
summable decay\footnote{Except of course in $d=2$, where we lose a logarithm, but anyway this is hidden by the stochastic domination notation.} of $j^{-d/2}$, and we obtain
\[
(\text{Id}-\lambda^2\mathcal{K})^{-1}(\mathcal{K}\chi_{r}(x_0))(x)
\lessapprox \lambda^2 r^2
\]
as desired.
\end{proof}

\printbibliography
\end{document}